\documentclass[11pt]{amsart}
\usepackage{amsxtra,amscd, color}
\usepackage{amssymb}
\usepackage{enumitem}
\theoremstyle{plain}
\newtheorem{theorem}{Theorem}[section]
\newtheorem{lemma}[theorem]{Lemma}
\newtheorem{definition-theorem}[theorem]{Definition-Theorem}
\newtheorem{proposition}[theorem]{Proposition}

\newtheorem{question}[theorem]{Question}
\newtheorem{corollary}[theorem]{Corollary}
\newtheorem{definition}[theorem]{Definition}
\newtheorem{example}[theorem]{Example}

\newtheorem{remark}[theorem]{Remark}
\newtheorem{remarks}[theorem]{Remarks}
\newtheorem{conjecture}[theorem]{Conjecture}

\newtheorem{Thm}{Theorem}[section]
\newtheorem{Prop}[Thm]{Proposition}
\newtheorem{Lem}[Thm]{Lemma}
\newtheorem{Cor}[Thm]{Corollary}

\theoremstyle{definition} % no italic text in these
\newtheorem{Def}[Thm]{Definition}
\newtheorem{Ex}[Thm]{Example}
\newtheorem{Rmk}[Thm]{Remark}
\newtheorem{Rmks}[Thm]{Remarks}
\newtheorem{Qtn}[Thm]{Question}
\def\ker{{\rm ker \,}}
\def\im{{\rm im \,}}

\def\char{{\rm char \,}}

\def\Lie{{\rm Lie \,}}
\def\co{{\rm co \,}}

\def\ZZ{\mathbb{Z}}

\def\mm{\mathbf{m}}

\def\OO{\mathcal{O}}

\def\gg{\mathfrak{g}}
\def\sl{\mathfrak{sl}}
\def\hh{\mathfrak{h}}
\newcommand \bth[1] { \begin{theorem}\label{t#1} }
\newcommand \ble[1] { \begin{lemma}\label{l#1} }

\newcommand \bpr[1] { \begin{proposition}\label{p#1} }
\newcommand \bqu[1] { \begin{question}\label{q#1} }
\newcommand \bco[1] { \begin{corollary}\label{c#1} }
\newcommand \bde[1] { \begin{definition}\label{d#1}\rm }
\newcommand \bex[1] { \begin{example}\label{e#1}\rm }
\newcommand \bre[1] { \begin{remark}\label{r#1}\rm }
\newcommand \bres[1] { \begin{remarks}\label{r#1}\rm }
\newcommand \bcj[1] { \begin{conjecture}\label{j#1}\rm }
\renewcommand {\eth} { \end{theorem} }
\newcommand {\ele} { \end{lemma} }

\newcommand {\epr} { \end{proposition} }
\newcommand {\equ} {\end{question} }
\newcommand {\eco} { \end{corollary} }
\newcommand {\ede} { \end{definition} }
\newcommand {\eex} { \end{example} }
\newcommand {\ere} { \end{remark} }
\newcommand {\eres} { \end{remarks} }
\newcommand {\ecj} { \end{conjecture} }
\newcommand {\enota} { \end{notation} }
\def \OO {{\mathcal{O}}}

\def \id { {\mathrm{id}} }

\def \Lie { {\mathrm{Lie \,}} }

\def \sl {\mathfrak{sl}}

\def \mm  {\mathfrak{m}}

\def \sl {\mathfrak{sl}}
\DeclareMathOperator \maxspec { {\mathrm{Maxspec}}}

\DeclareMathOperator \Ext { {\mathrm{Ext}} }
\DeclareMathOperator \GKdim {{\mathrm{GK \, dim}}}

\begin{document}

%%%%%%%%%%%%%%%%%%%%%%%%%%%%%%%%%%%%%%%%%%%%%%%%%%%%%%%%%%%%%%%%%%%%%%%%%%%
%%%%%%%%%%%%%%%%%%%%%%    Title    %%%%%%%%%%%%%%%%%%%%%%%%%%%%%%%%%%%%%%%%

\title[Duals of PI Hopf algebras]
{The finite dual of commutative-by-finite Hopf algebras}
\author[K. A. Brown]{K. A. Brown}
\thanks{The research of K. Brown was partially supported by Leverhulme Emeritus Fellowship EM-2017-081.}
\address{School of Mathematics and Statistics \\
University of Glasgow \\
Glasgow G12 8QQ, Scotland
}
\email{Ken.Brown@glasgow.ac.uk}
\author[M. Couto]{M. Couto}
\thanks{The PhD research of M. Couto was supported by a grant of the Portuguese Foundation for Science and Technology, SFRII/BD/102119/2014.}
\address{
}
\email{miguel.lourenco.couto@gmail.com}
\author[A. Jahn]{A. Jahn}
\thanks{The PhD research of A. Jahn was supported by a scholarship from the Carnegie Trust for the Universities of Scotland.}
\address{
}
\email{}
\date{}
\keywords{Hopf algebra, finite dual, polynomial identity}
\subjclass[2010]{Primary 16T05; Secondary 16T20, 16Rxx, 16P40}

\begin{abstract}
The finite dual $H^{\circ}$ of an affine commutative-by-finite Hopf algebra $H$ is studied. Such a Hopf algebra $H$ is an extension of an affine commutative Hopf algebra $A$ by a finite dimensional Hopf algebra $\overline{H}$. The main theorem gives natural conditions under which $H^{\circ}$ decomposes as a crossed or smash product of $\overline{H}^{\ast}$ by the finite dual $A^{\circ}$ of $A$. This decomposition is then further analysed using the Cartier-Gabriel-Kostant theorem to obtain component Hopf subalgebras of $H^{\circ}$ mapping onto the classical components of $A^{\circ}$. The detailed consequences for a number of families of examples are then studied.  
\end{abstract}
\maketitle

\section{Introduction}
\subsection{}\label{subsect1.1} This is the second in a series of papers about affine commutative-by-finite Hopf algebras and their duals, the first in the series being \cite{BrCou}. In that paper we defined a Hopf algebra $H$ over the algebraically closed field $k$ to be \emph{affine commutative-by-finite} if $H$ is a finitely generated module over a normal commutative finitely generated Hopf subalgebra $A$. (To say that $A$  is \emph{normal} in $H$ means that it is closed under the adjoint actions of $H$ on $A$, recalled at the start of \ref{PIdefn}.) Let $A^+$ denote the augmentation ideal of $A$. Normality of $A$ ensures that $A^+H$ is a Hopf ideal of $H$, so that $\overline{H} := H/A^+ H$ is a finite dimensional quotient Hopf algebra of $H$, and it is natural to view $H$ as an extension of $A$ by $\overline{H}$. 

Our aim in this paper is to obtain, for an affine commutative-by-finite Hopf algebra $H$, a description of the finite dual $H^{\circ}$ of $H$  which generalises the classical description valid when $H = A$ given by the Cartier-Gabriel-Kostant structure theorem \cite[Corollary 5.6.4, Theorem 5.6.5]{Mont}. Recall that in that case, assuming that $k$ has characteristic 0, $A^{\circ}$ is a smash product $U(\mathfrak{g}) \sharp kG$, where $A$ is the algebra of polynomial functions $\mathcal{O}(G)$ of the affine algebraic group $G$, $U(\mathfrak{g})$ is the enveloping algebra of its Lie algebra $\mathfrak{g}$ and $kG$ is the group algebra of $G$. The description we obtain combines (generalisations of) these ingredients with the dual $\overline{H}^{\ast}$ of $\overline{H}$. It comes closest to replicating the classical cocommutative picture when $H$ satisfies two additional hypotheses. Currently we know of no examples where these hypotheses fail to hold. We briefly explain the additonal hypotheses in $\S$\ref{subsect1.2}, then state the main theorem in $\S$\ref{subsect1.3}. 

\medskip     

\subsection{Coideal splitting and orbital semisimplicity}\label{subsect1.2} Let $H$ and its Hopf subalgebra $A$ satisfy the hypotheses of $\S$\ref{subsect1.1}, and assume that $A$ is semiprime or central, or that $H$ is pointed. It is known, as we'll recall in Theorem \ref{properties}, that $A$ is a direct summand of $H$ as right $A$-modules, say 
\begin{equation}\label{split} H \; =  \; A \oplus X.
\end{equation}
Our first additional hypothesis is:

$$ \textbf{(CoSplit)}\quad  X \textit{ in (\ref{split}) can be chosen to be a coideal of } H. $$

To explain the second hypothesis, note first that the left adjoint action of $H$ on $A$ factors through $\overline{H}$, so that for each maximal ideal $\mathfrak{m}$ of $A$ we can define    $\mathfrak{m}^{(\overline{H})}$ to be the unique largest $\overline{H}$-invariant ideal of $A$ contained in $\mathfrak{m}$. By work of Skryabin \cite{Sk10}, which we recall in $\S$\ref{orbitsec}, $A/\mathfrak{m}^{(\overline{H})}$ is a finite dimensional algebra. The second hypothesis is:

$$ \textbf{(OrbSemi)} \quad \textit{For every }\mathfrak{m} \in\mathrm{Maxspec}A, \; A/\mathfrak{m}^{(\overline{H})} \, \textit{is semisimple.}$$ 

\medskip

\subsection{Main theorem}\label{subsect1.3} Here is the main result stated here in the simplest case, namely when $k$ has characteristic 0; the full result is Theorem \ref{crux}. The symbol $R\#T$ is used to denote the smash product of the $k$-algebra $R$ by the Hopf algebra $T$, see \cite[Definition 4.1.3]{Mont}.

\begin{theorem}\label{main} Let $k$ have characteristic 0 and let $H$ be an affine commutative-by-finite Hopf $k$-algebra, finite over the normal commutative Hopf subalgebra $A$. (So $A = \mathcal{O}(G_H)$ for an affine algebraic $k$-group $G_H$ with Lie algebra $\mathfrak{g}$.)  Assume that $A \subseteq H$ satisfies $\mathrm{\mathbf{(CoSplit)}}$ and $\mathrm{\mathbf{(OrbSemi)}}$. Denote the finite dual of $H$ by $H^{\circ}$ and recall that $\overline{H}$ denotes the finite dimensional Hopf algebra $H/A^+ H$.

\begin{enumerate}
\item As a vector space, $H^\circ$ decomposes as $U(\mathfrak{g}) \otimes \overline{H}^{\ast} \otimes kG_H$; as an algebra it decomposes as a smash product,

\begin{equation}\label{finalHdual}
H^\circ \;\cong \; \overline{H}^{\ast} \#\, A^\circ \;\cong \; (\overline{H}^{\ast} \# \,U(\gg))\# kG_H.
\end{equation}

\item $H^\circ$ contains three distinguished Hopf subalgebras, namely $\overline{H}^{\ast}$, a differential smash product
\begin{equation}\label{finalWdecomp}
W(H^\circ) \; \cong \;\overline{H}^{\ast} \#\, U(\gg),
\end{equation}
and a skew group algebra
\begin{equation}\label{finalkGdecomp}
\widehat{kG_H}\; \cong \; \overline{H}^{\ast} \#\, kG_H.
\end{equation}
\item Suppose in addition that $\overline{H}$ is semisimple. Then the action of $U(\gg)$ on $\overline{H}^{\ast}$ is inner, so that
\begin{equation}\label{coup} H^{\circ} \; \cong \; (\overline{H}^{\ast} \otimes U(\gg)) \# kG_H,
\end{equation}
a skew group algebra with coefficient ring the Hopf subalgebra $W(H^{\circ})$.
\end{enumerate}
\end{theorem}

We currently know no counterexample to the possibility that, given any affine commutative-by-finite Hopf algebra $H$, there is a choice of normal commutative Hopf subalgebra $A \subseteq H$ which satisfies $\mathbf{(CoSplit)}$ and $\mathbf{(OrbSemi)}$. This is recorded as Question \ref{hyps}. 

Versions of some of the results in this paper appeared in the PhD theses of the second and third authors, \cite{Cou} and \cite{Jahn}. Since both these theses are publicly available online, some detailed calculations from them are not repeated here; full references are given where this occurs. 

The preprint \cite{LL}, which appeared while the present paper was being written, contains presentations by generators and relations of the finite duals $H^{\circ}$ of the affine prime regular Hopf algebras $H$ of Gel'fand-Kirillov dimension one (all of which are commutative-by-finite).  The results of \cite{LL} and the present paper are largely complementary, the only overlap being with $\S$\ref{GKdim1} of this paper, where the implications of Theorem \ref{crux} for prime regular affine Hopf algebras of GK-dimension one are examined.  

\medskip

\subsection{Layout of the paper}\label{plan} The final $\S$\ref{Examples} of the paper reviews  a number of families of affine commutative-by-finite Hopf algebras. We discuss there the validity of the hypotheses of Theorem \ref{crux} for each family, and the consequences for the structure of the Hopf dual for the algebras in each family. The reader may find it helpful to review some of these examples in tandem with the rest of the paper.

There is a brief review of the properties of affine commutative-by-finite Hopf algebras $H$ in $\S$\ref{OurClass}. $\S$\ref{FinDual} is devoted to the proof of our first key result, Theorem \ref{decompdual}, which obtains sufficient conditions for $H$ to decompose as a crossed or smash product of $\overline{H}^{\ast}$ by $A^{\circ}$. $\S\S$\ref{Wsect} and \ref{kGsec} are concerned respectively with the two Hopf subalgebras $W(H^{\circ})$ and $\widehat{kG_H}$ of $H^{\circ}$ which feature in Theorem \ref{main}(2), and which are respectively generalisations from the commutative case of $U(\mathfrak{g})$ (in characteristic 0) and $kG_H$. It is in analysing the structure of $\widehat{kG_H}$ that the hypothesis $\mathbf{(OrbSemi)}$ comes into play - this is introduced in $\S$\ref{sectionorbss}. The very brief $\S$\ref{final} contains the full statement of Theorem \ref{crux} and its deduction from the results of $\S\S$\ref{FinDual}-\ref{kGsec}.

\subsection{Notation}\label{notation} Throughout this paper $k$ will denote an algebraically closed field, all vector spaces are over $k$ unless stated otherwise, and all unadorned tensor products are over $k$. The \emph{Gelfand-Kirillov dimension} of an algebra $R$ will be denoted by $\GKdim R$; see \cite{KL} for details. For a Hopf algebra $H$ we use the usual notation of $\Delta,\epsilon$ for the coalgebra structure, with $\Delta(h)=\sum h_1\otimes h_2$ for $h \in H$, and we use $S$ to denote its antipode. The augmentation ideal $\ker\epsilon$ of a Hopf algebra $H$ will always be denoted by $H^+$. Further notation will be introduced in $\S$\ref{stand}. 

\bigskip

\section{Commutative-by-finite Hopf algebras}\label{OurClass}

\noindent In this section we briefly recall the most basic properties of affine commutative-by-finite Hopf algebras. At the same time notation will be fixed which will remain in force throughout the paper. Full details of the results stated here can be found in \cite{BrCou}.

\subsection{Definition and basic properties}\label{PIdefn} Recall \cite[\textsection 3.4]{Mont} that a subalgebra $K$ of a Hopf algebra $H$ is \emph{normal} if it is invariant under the left and right adjoint actions of $H$; that is, for all $k\in K$ and $h\in H$, 
$$ (\mathrm{ad}_{\ell} h)(k)= \sum h_1 k S(h_2) \in K \qquad \text{and} \qquad (\mathrm{ad}_{r} h)(k)= \sum S(h_1)kh_2 \in K. $$

The following class of Hopf algebras will be the focus of this paper.

\begin{Def}\label{almost} A Hopf $k$-algebra $H$ is \emph{affine commutative-by-finite} if it is a finite (left or right) module over a commutative normal Hopf subalgebra $A$ which is finitely generated as a $k$-algebra.
\end{Def}

For full proofs and references for the following see \cite[$\S$2]{BrCou}. 

\begin{Thm}\label{properties}
Let $H$ be an affine commutative-by-finite Hopf algebra, finite over the normal commutative Hopf subalgebra $A$ with augmentation ideal $A^+$.
\begin{enumerate}
\item\label{affnoeth} $A$ and $H$ are Noetherian $k$-algebras.
\item $A^+ H = HA^+$ is a Hopf ideal of $H$, and 
$$ \overline{H} \; := \; H/A^+ H $$
is a finite dimensional quotient Hopf algebra of $H$. We denote the Hopf algebra surjection by $\pi: H \longrightarrow \overline{H}$.
\item The left adjoint action of $H$ on $A$ factors through $\overline{H}$, so that $A$ is a left $\overline{H}$-module algebra.
\item The antipode $S$ of $H$ is bijective.
\item $H$ is a finite module over its centre $Z(H)$, which is affine.
\item $ \GKdim H = \GKdim A < \infty. $
\end{enumerate}
\noindent Assume for (7)-(10) that $A$ is semiprime or central, or that $H$ is pointed.
\begin{enumerate}
\item[(7)] $H$ is a finitely generated projective generator as left and right $A$-module.
\item[(8)] $A$ is a left (resp. right) $A$-module direct summand of $H$.
\item[(9)] $A\subseteq H$ is a faithfully flat $\overline{H}$-Galois extension.
\item[(10)] $A$ equals the right and the left $\overline{H}$-coinvariants of the $\overline{H}$-comodule $H$; that is,
$$ H^{\co\pi} \; = \; \,^{\co\pi}H \; = \; A.$$
\end{enumerate}
\end{Thm}

\begin{Rmks}\label{propremarks} Keep the notation and hypotheses of Theorem \ref{properties}.

\noindent$(1)$ (Radford \cite{Rad80}) In general, $H$ is \emph{not} a free $A$-module. For example, let $H=\OO(SL_2(k))$. This commutative Hopf algebra is a finite module over the Hopf subalgebra $A$ generated by the monomials of even degree, but it is not a free $A$-module.
\medskip
\noindent$(2)$ An important setting where $H$ \emph{is} $A$-free is when $H$ decomposes as a crossed product $H = A \#_{\sigma} \overline{H}$. By a celebrated result of Doi and Takeuchi \cite{DT}, this happens if and only if there is a cleaving map $\gamma:\overline{H} \longrightarrow H$ - that is,  $\gamma$ is a convolution invertible right $\overline{H}$-comodule map. Moreover, given that the extension is $\overline{H}$-Galois by Theorem \ref{properties}(4), such a cleaving map exists if and only if $A \subseteq H$ has the \emph{normal basis property}.  For details, see \cite[Propositions 7.2.3, 7.2.7, Theorem 8.2.4]{Mont}. Such a decomposition of $H$ is guaranteed when the coradical of $H$ is contained in $AG(H)$, by \cite[Corollary 4.3]{Schn92}.
\medskip

\noindent$(3)$ Since Theorem \ref{properties}(4) follows from \cite[Corollary 2]{Sk06}, which is valid for all Noetherian Hopf algebras satisfying a PI,  in Definition \ref{almost} it is enough to require that $A$ is normal on one side only; see \cite[Lemma 4.10(1)]{BrCou}.
\medskip

\noindent$(4)$ In seeking to apply Theorem \ref{properties}(7)-(10) one can often reduce to the case where $A$ is semiprime. For recall that $A$ is always semiprime when $k$ has characteristic 0 \cite[11.6]{Water}. Moreover in characteristic $p > 0$ one can usually reach this setting by replacing A by its image under a suitable power of the Frobenius map; we leave the details to the interested reader.

\end{Rmks}
\medskip

\subsection{Further notation}\label{stand}
For brevity we shall henceforth record the set-up of Theorem \ref{properties}, including $\overline{H}$ and the map $\pi$,  as
$$A \subseteq H \textit{is an affine commutative-by-finite Hopf algebra.}$$
The nilrdical $N(A)$ of the commutative affine Hopf $k$-algebra $A$ is a Hopf ideal by \cite[Lemma 5.2(1)]{BrCou}. There is thus an affine algebraic $k$-group $G_H$ such that $A/N(A) \cong \mathcal{O}(G_H)$ \cite[11.7]{Water}. We will often identify $G_H$ with the space $\maxspec(A)$ of maximal ideals of $A$ and with the group of algebra homomorphisms from $A$ to $k$, that is with the group-like elements $G(A^{\circ})$ of the finite dual $A^{\circ}$. For $g \in G_H$, we denote the corresponding members of $\maxspec(A)$ and $G(A^{\circ})$ by $\mm_g$ and $\chi_g$.

\bigskip

%%%%%%%%%%%%%%%%%%%%%%%%%%%%%%%%%%%%%%
\section{The finite dual}\label{FinDual}

\subsection{Commutative recap}\label{commcase}When $A$ is a commutative affine Hopf $k$-algebra the structure of $A^{\circ}$ is well understood - it is cocommutative and decomposes  as a smash product
\begin{equation}\label{comm} A^{\circ} \; \cong \; A' \# kG, 
\end{equation}
by the classical structure theorem for cocommutative pointed Hopf algebras, due to Cartier-Gabriel-Kostant \cite[Theorem 5.6.4(3)]{Mont}. Here, $A'$ is a connected Hopf subalgebra of $A$, and, as discussed in $\S$\ref{stand}, $G = G(A^{\circ})$ is the group of characters of $A$.  Moreover, when $k$ has characteristic 0, $A'$  is the enveloping algebra $U(\mathfrak{g})$ of the Lie algebra $\mathfrak{g}$ of $G$, \cite[Theorem 5.6.5]{Mont}. That is, $\gg=(A^+/{A^+}^2)^{\ast}$, consisting of the primitive elements of $A^{\circ}$, the functionals vanishing at 1 and ${A^+}^2$. Thus, by \cite[Proposition 9.2.5]{Mont}, 
\begin{equation}\label{Lie} A'  = \{f \in A^{\circ} : f((A^+)^n) = 0, \text{ for some } n > 0 \}= U(\mathfrak{g}),
\end{equation}
where the first equality holds in general and the second when $k$ has characteristic 0.
Similarly, as can easily be checked, with no condition on the characteristic of $k$,
\begin{equation}\label{groupalg} kG=\lbrace f\in A^{\circ}: f(\mm_1\cap \ldots\cap \mm_r)=0 , r \geq 1, \mm_i\in\maxspec(A) \rbrace, 
\end{equation}
the functionals vanishing on some finite intersection of maximal ideals of $A$.

\medskip

\subsection{Subspaces and quotient spaces of $H^{\circ}$}\label{subspaces}The following lemma gathers together some basic facts concerning the components $A^{\circ}$ and $\overline{H}^{\ast}$ of $H^{\circ}$. Fix throughout the rest of the paper the Hopf embedding and the Hopf surjection  
$$\iota:A\hookrightarrow H \; \textit{ and } \; \pi:H \longrightarrow \overline{H},$$
and the projection of right $A$-modules along $X$, 
\begin{equation}\label{proj}\Pi:H = A \oplus  X \twoheadrightarrow A : h = a+x \mapsto a,
\end{equation}
afforded by Theorem \ref{properties}(8) when $A$ is semiprime or central, or $H$ is pointed.

\begin{Lem}\label{HbarHopf}
Let $A \subseteq H$ be an affine commutative-by-finite Hopf algebra and assume that $A$ semiprime or central, or $H$ is pointed. 
\begin{enumerate}
\item There is an embedding of Hopf algebras
$$ \pi^{\circ} : \overline{H}^{\ast}\hookrightarrow H^{\circ}: \beta \mapsto \beta \circ \pi. $$
\item $\pi^\circ(\overline{H}^{\ast})$, which we identify with $\overline{H}^{\ast}$, is a normal Hopf subalgebra of $H^{\circ}$. 
\item $H^{\circ}$ is a free (right and left) $\overline{H}^{\ast}$-module.
\item There is a map of right $A^{\circ}$-comodules
$$ \Pi^{\circ}: A^{\circ} \longrightarrow H^{\circ} : \alpha \mapsto \alpha \circ \Pi. $$
\item There is a Hopf algebra map 
$$ \iota^{\circ} : H^{\circ} \longrightarrow A^{\circ} : f \mapsto f \circ \iota, $$
with $\iota^{\circ} \circ \Pi^{\circ} = \id_{A^{\circ}}.$ Thus $\Pi^{\circ}$ is injective and $\iota^{\circ}$ is surjective.
\item Via $\iota^{\circ}$, $H^{\circ}$ is canonically a right and left $A^{\circ}$-comodule algebra, with coinvariants
$$ (H^{\circ})^{co\,\iota^{\circ}}= {^{co\,\iota^{\circ}}(H^{\circ})}=\overline{H}^*. $$
\item The Hopf ideal $(\overline{H}^{\ast})^+ H^{\circ}$ of $H^{\circ}$ is contained in $\ker\iota^{\circ}$, so that 
\begin{equation}\label{hope} (\overline{H}^{\ast})^+ H^{\circ} \oplus \Pi^{\circ}(A^{\circ}) \subseteq \ker\iota^{\circ} \oplus \Pi^{\circ}(A^{\circ}) = H^\circ.
\end{equation}
\end{enumerate}
\end{Lem}
\begin{proof} (1),(2) Define $T := \{f \in H^{\circ} : f(A^+H) = 0\}$. Since $A^+H$ is a Hopf ideal of $H$, it is clear that $T$ is a Hopf subalgebra of $H^{\circ}$ which is isomorphic as a Hopf algebra to $\overline{H}^*$, via the map $\pi^{\circ}: \overline{H}^{\ast} \longrightarrow H^{\circ}: \beta \mapsto \beta \circ \pi$.

Since $HA^+ = A^+ H$ by Theorem \ref{properties}(2), it is very easy to check that 
\begin{equation}\label{hunk}T=\lbrace f\in H^{\circ}: f(ah)=\epsilon(a)f(h)=f(ha), \forall a\in A, \forall h\in H \rbrace. 
\end{equation} 

To deduce normality of $T$, let $f\in\overline{H}^*$ and $\varphi\in H^{\circ}$. For any $a\in A^+, h\in H$,
\begin{eqnarray*}
(\mathrm{ad}_{r} \varphi)(f) (ah) &=& \left[ \sum S^\circ(\varphi_1)f\varphi_2 \right](ah) = \sum S^\circ(\varphi_1)(a_1h_1) f(a_2h_2) \varphi_2(a_3h_3) \\
&=& \sum \varphi_1(S(h_1)S(a_1)) \epsilon(a_2)f(h_2) \varphi_2(a_3h_3) = \sum \varphi(S(h_1)S(a_1)a_2h_3) f(h_2) \\
&=& \epsilon(a) \sum \varphi(S(h_1)h_3) f(h_2) =0.
\end{eqnarray*}
Similarly, one shows that $(\mathrm{ad}_{\ell}\varphi)(ha) = 0$ for $h \in H$ and $a \in A$, and hence $T$ is normal by (\ref{hunk}).

\medskip

(3) Since $\overline{H}^{\ast}$ is a finite-dimensional and normal Hopf subalgebra of $H^{\circ}$, $H^{\circ}$ is a free $\overline{H}^{\ast}$-module by \cite[Theorem 2.1(2)]{Schn}. 

\medskip

(4),(5) Clearly, the restriction map $\iota^\circ: f \mapsto f \circ \iota$ is a Hopf algebra map whose image lies in $A^{\circ}$.

Let $f \in A^{\circ}$, so $f(J) = 0$ for an ideal $J$ of $A$ of finite codimension. Thanks to Theorem \ref{properties}(7) $H = A \oplus X$ for a right $A$-submodule $X$ of $H$, so the left ideal $HJ$ of $H$ decomposes as a right $A$-module,
\begin{equation}\label{decomp} HJ = J \oplus XJ \subseteq A \oplus X = H.
\end{equation}
In view of (\ref{decomp}) we can define $\widehat{f} \in H^{\ast}$ by $\widehat{f}(X) = 0$ and $\widehat{f}_{\mid A} = f$. Since $\widehat{f}(HJ) = 0$ and $\mathrm{dim}_k (H/HJ)  < \infty$, $\widehat{f} \in  H^{\circ}$ by \cite[Lemma 9.1.1]{Mont}. By construction, $f \circ  \Pi = \widehat{f}$; that is, $\Pi^{\circ}(f) = \widehat{f}$. It is clear that $\iota^{\circ} \circ \Pi^{\circ} = \mathrm{id}_{A^{\circ}}$, so $\Pi^{\circ}$ is injective and $\iota^{\circ}$ surjective. 

It remains to prove that $\Pi^{\circ}$ is a map of right $A^{\circ}$-comodules. The right $A^{\circ}$-comodule structure of $H^{\circ}$ is given by $\rho:=(\id\otimes \iota^{\circ})\Delta$. Let $f\in A^{\circ}$. Then for all $a\in A, h\in H$ we have
\begin{eqnarray*}
\left( (\rho\circ\Pi^\circ)(f)\right)(h\otimes a) &=& \left(\rho(f\circ\Pi)\right)(h\otimes a) = \left[ \sum (f\circ\Pi)_1 \otimes (f\circ\Pi)_2\circ\iota \right](h\otimes a) \\
&=& \sum (f\circ\Pi)_1(h) (f\circ\Pi)_2 (a) = (f\circ\Pi)(ha) = f(\Pi(h)a).
\end{eqnarray*}
On the other hand,
$$ \left((\Pi^\circ\otimes \id)\Delta(f)\right)(h\otimes a) = \left[\sum (\Pi^\circ\circ f_1) \otimes f_2\right](h\otimes a) = \sum f_1(\Pi(h)) f_2(a) = f(\Pi(h)a), $$
so (4) is proved.

\medskip

(6) $H^{\circ}$ is canonically a right and left $A^{\circ}$-comodule algebra, the right comodule map being $\rho:=(\id\otimes\iota^{\circ})\circ\Delta$. The right coinvariants are the maps $f$ such that 
$$\rho(f)=\sum f_1\otimes (f_2\circ \iota)= f\otimes \epsilon_A, $$ 
so that $f(ha)=f(h)\epsilon(a)$ for all $h\in H, a\in A$. By (\ref{hunk}), these are precisely the maps in $\overline{H}^*$. The left case is analogous. 

\medskip

(7) First, $(\overline{H}^{\ast})^+ H^{\circ}$ is a Hopf ideal of $H^{\circ}$ since, by (2),  $\overline{H}^{\ast}$ is a normal Hopf subalgebra of $H^{\circ}$, so \cite[$\S$3.4]{Mont} applies. Let $f\in (\overline{H}^*)^+$. Then $f(A^+H)=0$ by definition of $\overline{H}^{\ast}$, and $\epsilon_{H^{\circ}}(f)=f(1)=0$. Thus  $f(A)=0$, so that $f \in \ker\iota^{\circ}$.  Hence $(\overline{H}^{\ast})^+ H^{\circ} \subseteq \ker\iota^{\circ}$, and the inclusion in (\ref{hope}) follows. For the equality in (\ref{hope}), first observe that the sums are direct by part (5). Let $f \in H^{\circ}$. Then $(\Pi^{\circ} \circ \iota^{\circ})(f) \in \Pi^{\circ}(A^{\circ})$, and 
$$ \iota^{\circ}(f - (\Pi^{\circ} \circ \iota^{\circ})(f)) = 0,$$
so that $H^{\circ} = \ker\iota^{\circ} + \Pi^{\circ}(A^{\circ})$, as required.
\end{proof}

\medskip

\subsection{Decomposition of $H^{\circ}$}

Our first main result is the following. For the definition of a \emph{cleaving map}, see Remark \ref{propremarks}(2). Recall that Theorem \ref{properties}(8) guarantees that under the hypotheses of the theorem below $H = A \oplus X$ for a right $A$-submodule $X$ of $H$, and, as discussed in $\S$\ref{subsect1.1}, $A \subseteq H$ satisfies $\mathbf{(CoSplit)}$ if $X$ can be chosen to be a coideal of $H$.

\begin{Thm}\label{decompdual}
Let $A \subseteq H$ be an affine commutative-by-finite Hopf $k$-algebra with $A$ semiprime or central, or $H$ pointed. 
\begin{enumerate}
\item Suppose that $H$ satisfies $\mathbf{(CoSplit)}$. Then $H^{\circ}$ decomposes as a smash product
\begin{equation}\label{smashed}
H^{\circ}\; \cong \; \overline{H}^{\ast} \# A^{\circ}, 
\end{equation}
the isomorphism (\ref{smashed}) being as left $\overline{H}^*$-modules and right $A^{\circ}$-comodules as well as of algebras. 

\item Suppose that as an algebra $H=A\#_\sigma \overline{H}$ is a crossed product whose cleaving map $\gamma: \overline{H}\to H$ is a coalgebra map. (For example, this is the case if $H$ decomposes as a smash product $H \cong A \# \overline{H}$ with $\overline{H}$ a Hopf subalgebra of $H$.) Then (1) applies and the action of $A^\circ$ on $\overline{H}^*$ is trivial, so that 
\begin{equation}\label{mashed} H^\circ \; \cong \; \overline{H}^{\ast}\otimes A^\circ 
\end{equation} 
as left $\overline{H}^*$-modules, right $A^\circ$-comodules and algebras.

%\item Suppose $X$ can be chosen to be an $A$-bimodule right (or, equivalently, left) ideal of $H$. Then $H^{\circ}$ decomposes as a crossed product,
%%%the isomorphism (\ref{cross}) being as left $\overline{H}^*$-modules, right $A^{\circ}$-comodules and algebras, for an action of $A^{\circ}$ on $\overline{H}^*$ and cocycle $\sigma$. Moreover, the above isomorphism induces a coalgebra isomorphism $H^{\circ}\cong\overline{H}^*\otimes A^{\circ}$.

%\item In case (1), $\overline{H}^{\ast}$ is embedded as the coefficient subalgebra of the smash or crossed product via $\pi^{\circ}$, and the kernel of $\iota^{\circ}:H^{\circ}\longrightarrow A^{\circ}$ is $(\overline{H}^{\ast})^+H^{\circ}$. 
\end{enumerate}
\end{Thm}
\begin{proof} (1)  Since $X=\ker\Pi$ is a coideal of $H$, $\Pi$ is a coalgebra map, and so the right $A^{\circ}$-comodule map $\Pi^{\circ}$ of Lemma \ref{HbarHopf}(4) is an algebra map. In particular, it is a cleaving map with convolution inverse $\Pi^{\circ}\circ S_{A^{\circ}}$. By Remark \ref{propremarks}(3) $ H^{\circ}$ is isomorphic to $\overline{H}^* \#_\sigma A^{\circ} $ for some action of $A^{\circ}$ on $\overline{H}^*$ and cocycle $\sigma$. Since the cleaving map $\Pi^{\circ}$ is an algebra map, the cocycle $\sigma$ is trivial. By \cite[Proposition 7.2.3]{Mont} the isomorphism carries the structures as stated.

(2) Since $\gamma$ is a coalgebra map, $H=A\oplus A\gamma(\overline{H}^+)$ and $A\gamma(\overline{H}^+)$ is a coideal of $H$, so (1) applies. It remains to prove the action is trivial, that is $ \Pi^\circ(f')\pi^\circ(f) = \pi^\circ(f)\Pi^\circ(f'),$ for all $f\in\overline{H}^*, f'\in A^\circ$. 

Let $f\in\overline{H}^*, f'\in A^\circ, a\in A, \bar{h}\in \overline{H}$. We have
\begin{eqnarray*}
\pi^\circ(f)\Pi^\circ(f')(a\gamma(\bar{h})) &=& \sum \pi^\circ(f)(a_1\gamma(\bar{h_1}))\Pi^\circ(f')(a_2\gamma(\bar{h_2}))\\ 
&=& \sum f\pi(a_1\gamma(\bar{h_1}))f'\Pi(a_2\gamma(\bar{h_2})) \\
&=& \sum \epsilon(a_1)f\pi\gamma(\bar{h_1})f'(a_2)\epsilon(\gamma(\bar{h_2}))\\ 
&=& \sum \epsilon(a_1)f(\bar{h_1})f'(a_2)\epsilon(\bar{h_2}) \\
&=& f(\bar{h})f'(a),
\end{eqnarray*}
since $\pi\circ\gamma = \id$, which follows from $\gamma$ being a right $\overline{H}$-comodule coalgebra map. The proof for $\Pi^\circ(f')\pi^\circ(f)$ is completely analogous.

\end{proof}

\begin{Rmks}\label{decomprmks}(1) In general the decompositions (\ref{smashed}) and (\ref{mashed}) are not isomorphisms of coalgebras. For instance, consider the group algebra $H=kD$  of the infinite dihedral group $D := \langle a,b: a^2=1, aba=b^{-1} \rangle$, with $k$ of characteristic 0. It decomposes into $A\oplus X$ with $A = k\langle b \rangle$ and $X = A(a-1)$, so $X$ is a coideal $A$-module and (\ref{smashed}) applies. Here, 
$$ A^\circ \cong k[f]\otimes k(k^\times,\cdot),$$ 
where the algebra homomorphism indexed by $\lambda\in k^\times$ is $\chi_\lambda$, with $\chi_{\lambda}(b) = \lambda$. Embedding $A^{\circ}$ into $H^{\circ}$ using the algebra homomorphism $\Pi^{\circ}$, (\ref{smashed}) yields the algebra structure
$$ H^{\circ} \; = \; kC_2 \otimes k[\Pi^{\circ}(f)] \otimes k\{\Pi^{\circ}(\chi_{\lambda}) : \lambda \in k^{\times}\}.$$
However the functional $\Pi^{\circ}(f)$ is \emph{not} primitive and the functionals $\Pi^{\circ}(\chi_\lambda)$ are \emph{not} group-like (although they are units in $H^{\circ}$). Namely, one calculates: 
$$ \Delta(\Pi^{\circ}(f))=\Pi^{\circ}(f)\otimes 1 + \alpha\otimes \Pi^{\circ}(f); $$ 
and 
$$ \Delta(\Pi^{\circ}(\chi_\lambda)) = \frac{1}{2} (\Pi^{\circ}(\chi_\lambda) \otimes \Pi^{\circ}(\chi_\lambda))  \left( (1+\alpha)\otimes 1 + (1-\alpha)\otimes\Pi^{\circ}(\chi_{\lambda^{-2}}) \right) , $$ 
where the group-like involution $\alpha$ is the generator of $\overline{H}^*=kC_2$. For details, see \cite[Corollary 6.7, Remark 6.8]{Jahn} or \cite[Corollary 4.4.6(II) and Appendix A.2]{Cou}; there is also an extensive analysis of $kD^{\circ}$, arriving at the same formulae, in \cite{GL}.

\medskip
\noindent(2) Note that the action of the smash product in (\ref{smashed}) is given by $$ f\cdot\varphi = \sum \Pi^\circ(f_1)\varphi \Pi^\circ(S_{A^\circ} f_2), $$ for any $f\in A^\circ, \varphi\in\overline{H}^*$. In particular, when $k$ has characteristic 0, using the notation (\ref{comm}) and (\ref{Lie}), $\Pi^\circ(\gg)\cong\gg$ and, if $f\in\gg$, then it acts as a derivation on $\overline{H}^*$ (that is, $f\cdot \varphi = \Pi^\circ(f)\varphi - \varphi\Pi^\circ(f)$), even though in general $\Pi^\circ(f)$ won't be a primitive element of $H^\circ$. Similarly, if $\chi_g\in A^\circ$ denotes the algebra homomorphism corresponding to $g\in G_H$, then $\chi_g$ acts by conjugation over $\overline{H}^*$ (i.e. $\chi_g\cdot \varphi = \Pi^\circ(\chi_g)\varphi\Pi^\circ(\chi_{g^{-1}})$), despite in general $\Pi^\circ(\chi_g)$ not being a grouplike in $H^\circ$.

\medskip

\noindent(3) The following question seems on the face of it to be rather optimistic, but we currently do not know any example of an affine commutative-by-finite Hopf algebra $H$ for which there is no choice of the subalgebra $A$ such that the answer is positive for $A \subseteq H$.

\begin{Qtn}\label{smash} When $A \subseteq H$ is affine commutative-by-finite with $A$ semiprime or central, or $H$ pointed, is $H^{\circ}$ always a smash product of $\overline{H}^{\ast}$ by $A^{\circ}$ for a suitable choice of $A$?
\end{Qtn}

\noindent In $\S$\ref{Examples} it is shown that Question \ref{smash} has a positive answer for many classes of commutative-by-finite Hopf algebras. 
\end{Rmks}

\bigskip

%%%%%%%%%%%%%%%%%%%%%%%%%%%%%%%%%%%%%%%%%%%%%%%%%%%
\section{The tangential component $W(H^{\circ})$}\label{Wsect}

\subsection{Definition and basic properties} Throughout this section the hypotheses and notation of $\S$\ref{stand} remain in force. We define an analogue for the finite dual of a commutative-by-finite Hopf algebra of the irreducible component $A'$ appearing in (\ref{comm}), as follows.

\begin{Def}\label{Wdef}
Let $A \subseteq H$ be an affine commutative-by-finite Hopf $k$-algebra. The \emph{tangential component} of $H^{\circ}$ is
$$ W(H^{\circ}) := \lbrace f\in H^{\circ}: f((A^+H)^n)=0, \text{ for some } n>0 \rbrace. $$ 
\end{Def}

\begin{Lem}\label{Wlemma}
Retain the notation and hypotheses of Definition \ref{Wdef}. 
\begin{enumerate}
\item $W(H^{\circ})$ is a normal Hopf subalgebra of $H^{\circ}$.
\item $\overline{H}^*$ is a normal Hopf subalgebra of $W(H^{\circ})$.
\end{enumerate}
\end{Lem}

\begin{proof} First, $W(H^{\circ})$ is a Hopf subalgebra because $A^+H$ is a Hopf ideal of $H$ and, for any Hopf ideal $I$ of $H$,
$$ \{ f \in H^{\circ} : f(I^n) = 0 \textit{ for some } n > 0 \} $$
is a Hopf subalgebra of $H^\circ$ by \cite[Lemma 9.2.1]{Mont}. Clearly, $W(H^{\circ})$  contains $\overline{H}^*$, by the definition of the latter subalgebra in Lemma \ref{HbarHopf}(1). Since $\overline{H}^{\ast}$ is normal in $H^{\circ}$ by Lemma \ref{HbarHopf}(2), it is \emph{a fortiori} normal in the Hopf subalgebra $W(H^{\circ})$. 

To prove normality of $W(H^\circ)$ let $\varphi\in H^{\circ}$ and $f\in W(H^{\circ})$, so $f$ vanishes at $(A^+H)^n = (A^+)^n H$ for some positive integer $n$. Take $a_1,\ldots, a_n\in A^+$ and $h\in H$. Then
\begin{eqnarray*}
&& \left[\sum S^\circ(\varphi_1)f\varphi_2\right](a_1\ldots a_nh) \\
&=& \sum S^\circ(\varphi_1)(a_{1,1}\ldots a_{n,1}h_1)f(a_{1,2}\ldots a_{n,2}h_2)\varphi_2(a_{1,3}\ldots a_{n,3}h_3) \\
&=&\sum \varphi_1(S(h_1)S(a_{n,1})\ldots S(a_{1,1}))f(a_{1,2}\ldots a_{n,2}h_2)\varphi_2(a_{1,3}\ldots a_{n,3}h_3) \\
&=& \sum \varphi(S(h_1)S(a_{n,1})\ldots S(a_{1,1})a_{1,3}\ldots a_{n,3}h_3)f(a_{1,2}\ldots a_{n,2}h_2) \\
&=& \sum \varphi(S(h_1)S(a_{n,1})a_{n,3}\ldots S(a_{1,1})a_{1,3} h_3)f(a_{1,2}\ldots a_{n,2}h_2),
\end{eqnarray*}
since $A$ is commutative. For each $i$ we write $a_{i,2}=\epsilon(a_{i,2})1+a_{i,2}'$, where $a_{i,2}'\in A^+$. Writing $[1,n]$ for $\lbrace 1,2,\ldots,n \rbrace$, we now have $$ f(a_{1,2}\ldots a_{n,2}h_2)=\sum_{\mathcal{P}\subseteq [1,n]}  \prod_{i\in\mathcal{P}}\epsilon(a_{i,2}) f\left(\prod_{i\in[1,n]\setminus\mathcal{P}}a_{i,2}'h_2\right).  $$ %= \sum_{k_1,\ldots,k_n} \prod_{i:k_i=1} \epsilon(a_{i,2}) f\left( \prod_{i:k_i=0} a_{i,2}' \, h_2\right)
Therefore,
\begin{eqnarray*}
&&\left[\sum S^\circ(\varphi_1)f\varphi_2\right](a_1\ldots a_n h) \\
&=& \sum \sum_{\mathcal{P}\subseteq[1,n]} \varphi(S(h_1)S(a_{n,1})a_{n,3}\ldots S(a_{1,1})a_{1,3} h_3) \prod_{i\in\mathcal{P}} \epsilon(a_{i,2}) f\left(\prod_{i\in[1,n]\setminus\mathcal{P}} a_{i,2}' \, h_2\right) \\
&=& \sum \sum_{\mathcal{P}\subseteq[1,n]} \varphi \left(S(h_1) \prod_{i\in\mathcal{P}} S(a_{i,1})\epsilon(a_{i,2})a_{i,3} \prod_{i\in[1,n]\setminus\mathcal{P}} S(a_{i,1})a_{i,3} \, h_3 \right) f\left(\prod_{i\in[1,n]\setminus\mathcal{P}} a_{i,2}' h_2\right) \\
&=& \sum \sum_{\mathcal{P}\subseteq[1,n]} \prod_{i\in\mathcal{P}} \epsilon(a_i) \, \varphi \left(S(h_1) \prod_{i\in[1,n]\setminus\mathcal{P}} S(a_{i,1})a_{i,3} \, h_3 \right) f\left(\prod_{i\in[1,n]\setminus\mathcal{P}} a_{i,2}' h_2\right).
\end{eqnarray*}
The summands where $\mathcal{P}\neq\emptyset$ vanish due to $\epsilon(a_i)=0$ for all $i\in[1,n]$ and the summands where $\mathcal{P}=\emptyset$ vanish because $f$ is zero on $(A^+H)^n={A^+}^nH$. The proof for the left adjoint action is similar, so  normality of $W(H^{\circ})$ in $H^{\circ}$ is proved. 
\end{proof}

\subsection{Decomposition of $W(H^\circ)$}

To analyse the structure of $W(H^{\circ}) $ we use the Hopf algebra surjection $\iota^{\circ}$ from $H^{\circ}$ to $A^{\circ}$ of Lemma \ref{HbarHopf}(5), given by restriction of functionals from $H$ to $A$. In the following result we decompose $W(H^\circ)$ into the crossed product $\overline{H}^* \#_\sigma A'$. Notice that the isomorphism (\ref{Wcrossed}) below does not require hypothesis $\mathrm{\mathbf{(CoSplit)}}$, in contrast to Theorem \ref{decompdual}. 

\begin{Thm}\label{decompW} Let $A \subseteq H$ be an affine commutative-by-finite Hopf algebra with $A$ semiprime or central, or $H$ pointed. Recall that  $A^{\circ}$ decomposes as $A' \#kG_H$ as in (\ref{comm}).
\begin{enumerate}
\item $\iota^{\circ}(W(H^{\circ})) = A'$.
\item $W(H^{\circ})$ decomposes as a crossed product for a cocycle $\sigma$ and an action of $A'$ on $\overline{H}^*$. That is, as algebras, left $\overline{H}^{\ast}$-modules and right $A'$-comodules,
\begin{equation}\label{Wcrossed} W(H^{\circ})\cong \overline{H}^* \#_\sigma A'. 
\end{equation}
\item $W(H^{\circ})^{\co \iota^\circ_{\mid_{W(H^\circ)}}} = \overline{H}^{\ast}.$
\item Suppose that $k$ has characteristic 0. Then $\ker \iota^\circ_{\mid_{W(H^\circ)}} = (\overline{H}^{\ast})^+ W(H^{\circ}).$
\item Suppose that $A \subseteq H$ satisfies $\mathrm{\mathbf{(CoSplit)}}$. Then $\sigma$ is trivial - that is, (\ref{Wcrossed}) is a decomposition of $W(H^{\circ})$ as a smash product.  When $k$ has characteristic 0, so that $A' = U(\mathfrak{g})$, the action of $U(\gg)$ on $\overline{H}^*$ is given by  
$$ f \cdot \varphi = \Pi^\circ(f) \varphi - \varphi\Pi^\circ(f), $$ 
for $f\in\gg, \varphi\in\overline{H}^*$. That is, $W(H^\circ)$ is a differential operator ring over $\overline{H}^*$.
\end{enumerate}
\end{Thm}

\begin{proof}

(1) Since $\iota^{\circ}$ is just restriction of functionals to the subdomain $A$, it is clear that $\iota^{\circ}(W(H^{\circ})) \subseteq A'$. To see that $\im \iota^{\circ}_{\mid_{W(H^\circ)}} = A'$, note the definition (\ref{Lie}) of $A'$, and then observe that, in the proof of Lemma \ref{HbarHopf}(4), if $f \in A' \subseteq A^{\circ}$, then the map $\widehat{f} \in H^{\circ}$ constructed there, with $\Pi^{\circ}(f) = \widehat{f}$, is actually in $W(H^{\circ})$.  Since $\iota^{\circ} \circ \Pi^{\circ} = \id_{A^{\circ}}$ by Lemma \ref{HbarHopf}(5), $\iota^{\circ}(\widehat{f}) = f$ and the result follows.

(2),(3): By Lemma \ref{HbarHopf}(6), $H^{\circ}$ is a right $A^{\circ}$-comodule algebra with $\rho=(\id\otimes\iota^{\circ})\Delta$. In particular, since $W(H^{\circ})$ is a Hopf subalgebra of $H^{\circ}$ by Lemma \ref{Wlemma}(2), and using (1) of the present theorem,
\begin{equation} \label{comodule} \rho(W(H^{\circ}))\subseteq W(H^{\circ})\otimes \iota^{\circ}(W(H^{\circ})) = W(H^{\circ})\otimes A'.
\end{equation} 
That is, the structure of right $A^{\circ}$-comodule of $H^{\circ}$ restricts to a structure of right $A'$-comodule on $W(H^{\circ})$.

We have noted in the proof of (1) that the right $A^{\circ}$-comodule map $\Pi^{\circ}: A^{\circ} \longrightarrow H^{\circ}$ maps $A'$ into $W(H^{\circ})$. Thus, in light of (\ref{comodule}), $\Pi^{\circ}_{\mid_{A'}}: A' \longrightarrow W(H^{\circ})$ is an injection of right $A'$-comodules.

A linear map from a coalgebra to an algebra is convolution invertible if and only if its restriction to the coradical is convolution invertible, see \cite[Lemma 14]{Tak} or \cite[Lemma 5.2.10]{Mont}. But $A'$ is connected by its definition as the irreducible component of the coalgebra $A^{\circ}$ containing $\varepsilon$,  \cite[page 78]{Mont}; that is, its coradical is $k$. Therefore, in the language of Remark \ref{propremarks}(3), $\Pi^{\circ}_{\mid_{A'}}$ is a cleaving map and by the result of Doi and Takeuchi recalled in Remark \ref{propremarks}(3), 
$$ W(H^{\circ}) \; \cong  \; W(H^{\circ})^{\co \iota^{\circ}_{\mid_{W(H^\circ)}}} \#_\sigma A'. $$

Since the $A'$-comodule structure of $W(H^{\circ})$ is the restriction of the $A^{\circ}$-comodule structure of $H^{\circ}$, it follows from Lemma \ref{HbarHopf}(6) that 
$$ W(H^{\circ})^{\co \iota^\circ_{\mid_{W(H^\circ)}}} = W(H^{\circ}) \cap  (H^{\circ})^{\co\iota^\circ}  =  W(H^{\circ}) \cap  \overline{H}^* = \overline{H}^*. $$ 

(4) Since $\mathrm{char}k = 0$, $A' = U(\gg)$ by (\ref{Lie}). It is clear from (\ref{Wcrossed}) that 
$$ W(H^{\circ})/(\overline{H}^{\ast})^+ W(H^{\circ}) \quad \cong \quad U(\mathfrak{g}). $$
Moreover we know from Lemma \ref{HbarHopf}(7) that 
\begin{equation}\label{into} (\overline{H}^{\ast})^+W(H^{\circ})\subseteq \ker \iota^\circ_{\mid_{W(H^\circ)}}. 
\end{equation}
By (1), $\im \iota^\circ_{\mid_{W(H^\circ)}} = U(\mathfrak{g})$. Since $A$ is affine, $\mathfrak{g}$ is finite dimensional, and hence $U(\mathfrak{g})$ is a noetherian algebra. In particular $U(\mathfrak{g})$ cannot be isomorphic to a proper factor of itself, and hence (\ref{into}) must be an equality, proving (4).

(5) This follows from (2) and Theorem \ref{decompdual}(1), since the cocycle $\sigma$ in (2) is in this case trivial by the latter result. The last part is a special case of Remark \ref{decomprmks}(2).
\end{proof}

\medskip

\noindent Recall that, in characteristic 0, a finite dimensional Hopf algebra is semisimple if and only if cosemisimple \cite{LR1}, so the following improvement of Theorem \ref{decompW}(5) applies when $\overline{H}$ is semisimple if $k$ has characteristic 0.

\begin{Cor}\label{inner}  Let $A \subseteq H$ be an affine commutative-by-finite Hopf $k$-algebra with $A$ semiprime or central, or $H$ pointed. Retain the additional notation of Theorem \ref{decompW}. Assume that $A \subseteq H$ satisfies hypothesis $\mathrm{\mathbf{(CoSplit)}}$, and that $\overline{H}$ is cosemisimple. Then the action of $U(\mathfrak{g})$ on $\overline{H}^{\ast}$ is inner, so that, as algebras,
$$ W(H^{\circ}) \; \cong \; \overline{H}^{\ast} \otimes U(\mathfrak{g}). $$
\end{Cor}

\begin{proof} By Theorem \ref{decompW}(5) $W(H^{\ast})$ is a smash product of algebras, 
\begin{equation}\label{start} \overline{H}^{\ast} \# U(\mathfrak{g}).
\end{equation}
By our hypothesis, $\overline{H}^{\ast}$ is semisimple, and so by Jacobson's theorem, (see for example \cite[Example 6.1.6, Proposition 6.2.1 and discussion following]{Mont}), since $k$ is algebraically closed, $\mathfrak{g}$ acts by inner derivations on $\overline{H}^{\ast}$. That is, there is a copy $u(\mathfrak{g})$ of $\mathfrak{g}$ contained in $\overline{H}^{\ast}$, (where we view $\overline{H}^{\ast}$ as a Lie algebra via the commutator bracket), such that for all $x \in \mathfrak{g}$ and $h \in \overline{H}^{\ast}$, 
$$  x\cdot h \; =\; [u(x),\, h]. $$
Replacing the copy of $\mathfrak{g}$ appearing in (\ref{start}) with the isomorphic Lie algebra 
$$\widehat{\mathfrak{g}} \; := \; \sum_{i}k(x_i - u(x_i)), $$
where $\{x_i \}$ is a $k$-basis of $\mathfrak{g}$, we see that 
$$  W(H^{\circ}) \; = \;  \overline{H}^{\ast} \otimes U(\widehat{\mathfrak{g}}), $$
as claimed.
\end{proof}

%%%%%%%%%%%%%%%%%%%%%%%%%%%%%%%%%%%%%%%%%%%%%%%%%%%%%%%%%%%%

\bigskip

\section{The subcoalgebra $\widehat{kG_H}$}\label{kGsec}

\subsection{Orbits in $\maxspec(A)$}\label{orbitsec} Recall the description (\ref{groupalg}) in $\S$\ref{commcase} of the Hopf subalgebra $kG$ of $A^{\circ}$. In this section we extend this viewpoint to $H$, defining a subcoalgebra $\widehat{kG_H}$ of $H^{\circ}$ which is a Hopf subalgebra under mild additional hypotheses and which reduces to $kG_H$ when $H = A$. We saw in Theorem \ref{properties}(3) that the left adjoint action of $H$ on $A$ factors over the ideal $A^+ H$, so that $A$ is a left $\overline{H}$-module algebra. An ideal $I$ of $A$ is said to be \textit{$\overline{H}$-stable} if 
$$ (\mathrm{ad}_{\ell}h)(a)\in I, \; \text{for all } h\in\overline{H}, a\in I. $$ 
Since the extension $A\subseteq H$ is faithfully flat by Theorem \ref{properties}(4), $I$ is $\overline{H}$-stable if and only if $HI=IH$.

\begin{Def}\label{cored}
Let $T$ be a Hopf algebra and $R$ a left $T$-module algebra, with $V$ a subspace of $R$. The $T$-\emph{core} of $V$ is the subspace $$ V^{(T)} \; = \; \{ v \in V : t\cdot v \in V, \,  \forall t \in T \}.$$
\end{Def}

\noindent Note that if $V$ is a right or left ideal of $R$ in Definition \ref{cored} then so is $V^{(T)}$. The crucial parts (2) and (3) of the following result are essentially due to Skryabin \cite[Theorems 1.1, 1.3]{Sk10}. For the proofs of Proposition \ref{Skryab} and Corollary \ref{orbit}, see \cite[Proposition 4.4]{BrCou}.

\begin{Prop}\label{Skryab}
Let $A \subseteq H$ be an affine commutative-by-finite Hopf algebra and let $\mm\in\maxspec(A)$.
\begin{enumerate}
\item $\mm^{(\overline{H})}$ is the largest $\overline{H}$-stable ideal of $A$ contained in $\mm$.
\item $A/\mm^{(\overline{H})}$ is an $\overline{H}$-simple algebra - that is, the $\overline{H}$-core of each of its proper ideals is $\{0\}$.
\item $A/\mm^{(\overline{H})}$ is a Frobenius algebra. In particular,  $A/\mm^{(\overline{H})}$ is finite dimensional.
\end{enumerate}
\end{Prop}

\noindent Following still Skryabin, we can use Proposition \ref{Skryab}(2) to extend the fundamental idea that a group of automorphisms of the algebra $A$ defines orbits  in $\maxspec (A)$, to the present setting where the group is replaced by a Hopf algebra.

\begin{Def}\label{deforbit}
Let $A \subseteq H$ be an affine commutative-by-finite Hopf algebra.

\begin{enumerate}
\item For each $\mm\in\maxspec(A)$, its $\overline{H}$-\textit{orbit} is the set of maximal ideals with the same $\overline{H}$-core, $$\OO_{\mm} = \lbrace \mm'\in\maxspec(A): {\mm'}^{(\overline{H})} = \mm^{(\overline{H})} \rbrace. $$

\item Define an equivalence relation $\sim^{(\overline{H})}$ on $G_H = \maxspec(A)$ as follows: for $g,h\in G_H$,  $g\sim^{(\overline{H})} h$ if and only if $ \mm_g^{(\overline{H})}=\mm_h^{(\overline{H})}.$
\end{enumerate}
\end{Def}

\noindent It's clear that $\sim^{(\overline{H})}$ is an equivalence relation; but it is far from obvious, although it follows easily from Proposition \ref{Skryab}(2), that the $\overline{H}$-orbits have the following simple description.

\begin{Cor}\label{orbit}
Let $A \subseteq H$ be an affine commutative-by-finite Hopf algebra.
\begin{enumerate}
\item $\OO_{\mm} = \lbrace \mm'\in\maxspec(A): \mm^{(\overline{H})}\subseteq \mm' \rbrace$.

\item $\OO_{\mm}$ is finite.
\end{enumerate}
\end{Cor}

\medskip

\subsection{Orbital semisimplicity}\label{sectionorbss}The results in $\S$\ref{orbitsec} lead naturally to:

\begin{Def}\label{orbsemi} Let $A \subseteq H$ be an affine commutative-by-finite Hopf algebra. Then $A \subseteq H$ is called \emph{orbitally semisimple} if $A/\mm^{(\overline{H})}$ is semisimple for all $\mm\in\maxspec(A)$. In this case we write that $A \subseteq H$ satisfies $\mathbf{(OrbSemi)}$.
\end{Def}

Note that, in view of Corollary \ref{orbit}(1), $A \subseteq H$ is orbitally semisimple if and only if 
\begin{equation}\label{equisemi} \mm^{(\overline{H})} = \bigcap_{\mm'\in\OO_\mm} \mm'.
\end{equation}

Since we have used the \emph{left} adjoint action in Definition \ref{orbsemi}, one might choose to call the above property \emph{left} orbital semisimplicity; but in fact it is shown in \cite[Lemma 4.10(2)]{BrCou} that $A \subseteq H$ is left orbitally semisimple if and only if it is right orbitally semisimple.

\medskip

We know no examples of affine commutative-by-finite Hopf algebras that are not orbitally semisimple. When $A$ is central in $H$ orbital semsimplicity is trivial; the next result lists other positive cases. In characteristic 0 there are overlaps between these, thanks to the result of Etingof and Walton \cite{EW} implying that the action of $\overline{H}$ factors through a group algebra when $\overline{H}$ is semisimple and cosemisimple and $A$ is a domain. 

\begin{Thm}\label{orbhold}(\cite[Theorem 4.8]{BrCou})
The affine commutative-by-finite Hopf algebra $A\subseteq H$ satisfies $\mathbf{(OrbSemi)}$ in the following cases.
\begin{enumerate}
\item The adjoint action of $\overline{H}$ on $\maxspec(A)$ factors through a group.
\item $\overline{H}$ is cosemisimple.
\item $\overline{H}$ is involutory and either $\char k=0$ or $\char k>\dim A/\mm^{(\overline{H})}$ for all $\mm\in\maxspec(A)$.
\end{enumerate}
\end{Thm}

\noindent We discuss particular families of examples of affine commutative-by-finite Hopf algebras, and their status with respect to orbital semisimplicity, in $\S$\ref{Examples}.

\medskip

\subsection{The character component of $H^{\circ}$}\label{charcomp} Here is the definition we have been working towards:

\begin{Def}\label{charcompdef} Let $H$ be an affine commutative-by-finite Hopf algebra, finite over the normal commutative Hopf subalgebra $A$. Recall from $\S$\ref{stand} that $G_H$ denotes the affine algebraic $k$-group such that $A/N(A) \cong \mathcal{O}(G_H)$.
\begin{enumerate}
\item The \textit{character component} of $H^\circ$ is 
$$ \widehat{kG_H} = \left\{ f\in H^{\circ}: f\left( H \mm_{g_1}^{(\overline{H})} \cap \ldots \cap H \mm_{g_r}^{(\overline{H})} \right)=0, \text{ for some } g_1,\ldots,g_r\in G_H/\sim^{(\overline{H})} \right\}. $$
\item For $g \in G_H$, write $\widehat{g}$ to denote the subspace $(H/H \mm_g^{(\overline{H})})^{\ast}$ of $\widehat{kG_H}$.
\end{enumerate}
\end{Def}

\noindent Here are some basic properties of these spaces.

\begin{Lem}\label{kGprop} Let $A \subseteq H$ be an affine commutative-by-finite Hopf $k$-algebra. Retain the notation of $\S$\ref{stand} and Definition \ref{charcompdef}.
\begin{enumerate}
\item If $g,h \in G_H$ with $g \sim^{(\overline{H})} h$, then $\widehat{g} = \widehat{h}$.
\item $\widehat{kG_H}=\bigoplus_{g\in G_H/\sim^{(\overline{H})}} \; \widehat{g}$.
\item $\widehat{kG_H}$ is a subcoalgebra of $H^{\circ}$.
\item For all $g \in G_H$, $\widehat{g}$ is a subcoalgebra of $\widehat{kG_H}$.

\item For all $g \in G_H$, 
$$  \widehat{g} \supseteq \bigoplus_{h \sim^{(\overline{H})} g} \left( H/H \mm_h \right)^{\ast} ,$$
and
$$   \left( H/H \mm_g \right)^{\ast} \supseteq  \overline{H}^{\ast} \Pi^{\circ}(\chi_g) + \Pi^\circ(\chi_g)\overline{H}^{\ast}. $$
\item For $g \in G_H$,  $S^{\circ}(\widehat{g})\; =\; \widehat{g^{-1}}$.
\item $\widehat{kG_H}\cap W(H^{\circ})\; =\; \overline{H}^{\ast} \; =\;  \widehat{1_G}$.
\item Assume that $A$ is a domain and let $d_A(H)$ be the minimal number of generators of $H$ as a right or left $A$-module. For $g \in G_H$, 
\begin{equation}\label{size} \mathrm{dim}_k (H/H\mm_g) = \mathrm{dim}_k (H/\mm_g H) = \mathrm{dim}_k (\overline{H}),
\end{equation}
and 
\begin{equation}\label{hatsize} \mathrm{dim}_k(\widehat{g}) \; = \; \mathrm{dim}_k (A/\mm_g^{(\overline{H})}) \mathrm{dim}_k(\overline{H})  \leq d_A(H)\mathrm{dim}_k(\overline{H}).
\end{equation}
If $H$ is prime (which implies that $A$ is a domain), then 
\begin{equation}\label{extra} \mathrm{dim}_k(\widehat{g}) \; \leq \; \left(\mathrm{dim}_k (\overline{H})\right)^2.
\end{equation}
\end{enumerate}
\end{Lem}
\begin{proof}
(1) This follows from the definition of $\sim^{(\overline{H})}$.
\medskip

\noindent (2) If $g, h \in G_H$ with $\OO_{\mm_g} \neq \OO_{\mm_h}$, then the ideals $\mm_g^{(\overline{H})}$ and $\mm_h^{(\overline{H})}$ of $A$ are comaximal, by Corollary \ref{orbit}(1). The same is thus true of the ideals $H\mm_g^{(\overline{H})}$ and $H\mm_h^{(\overline{H})}$ of $H$. Therefore, if $g_1, \ldots , g_r \in G_H$ are representatives of distinct $\sim^{(\overline{H})}$-orbits, the Chinese Remainder Theorem shows that
$$ \left. H \middle/ \bigcap_{i=1}^r H\mm_{g_i}^{(\overline{H})} \right. \cong \bigoplus_{i=1}^r H/H\mm_{g_i}^{(\overline{H})}. $$ 
Taking duals on both sides proves (2).

\medskip

\noindent (3),(4): Let $g \in G_H$. That $\widehat{g}$ is a subcoalgebra of $H^{\circ}$ is simply a special case of the elementary fact that, for an ideal $I$ of finite codimension in an algebra $R$, $(R/I)^{\ast}$ is a subcoalgebra of $R^{\circ}$. So (3) and (4) follow from this observation together with (2).
\medskip

\noindent (5) The first inclusion is clear. For the second, note first that
\begin{equation}\label{cop} \Delta (H\mm_g) = \Delta(H)\Delta(\mm_g) \subseteq H\mm_{1_{G_H}} \otimes H + H \otimes H\mm_g, 
\end{equation}
by the definition of the coproduct in $A/N(A) = \OO(G_H)$, combined with the fact that $N(A)$ is a Hopf ideal of $A$, \cite[Lemma 5.2(1)]{BrCou}. Let $f \in \overline{H}^{\ast}$. Then, using (\ref{cop}),
$$ (f\Pi^{\circ}(\chi_g)) (H\mm_g) \subseteq f(H\mm_{1_{G_H}})\Pi^{\circ}(\chi_g)(H) + f(H)\Pi^{\circ}(\chi_g)(H\mm_g) = 0.$$
So $\overline{H}^{\ast}\Pi^{\circ}(\chi_g) \subseteq (H/H\mm_g)^{\ast}$. The proof of the remaining inclusion is similar.

\medskip

\noindent (6) Let $g \in G_H$. To prove that $S^{\circ}(\widehat{g})=\widehat{g^{-1}}$ it suffices to show that $S(\mm_g^{(\overline{H})}) = \mm_{g^{-1}}^{(\overline{H})}$. It is easy to see that 
$$ \mm_g^{(\overline{H})}=\mathrm{r-Ann}_A(H/\mm_gH)=\mathrm{l-Ann}_A(H/H\mm_g). $$ 
Since $S$ is an algebra anti-isomorphism of $H$ with $S(A)=A$, it follows that 
$$ S(\mm_g^{(\overline{H})}) = \mathrm{l-Ann}_A(H/HS(\mm_g)) = \mathrm{l-Ann}_A(H/H\mm_{g^{-1}}). $$ 
This proves (6).

\medskip

\noindent (7) The final equality is clear from the definitions, and this also shows that $\overline{H}^{\ast} \subseteq \widehat{kG_H}$. Thus, by Lemma \ref{Wlemma}(2),  $\overline{H}^{\ast}\subseteq W(H^{\circ}) \cap \widehat{kG_H}$. 

For the reverse inclusion, suppose $f \in H^{\circ}$ with $f$ vanishing on $(A^+H)^n={A^+}^nH$ and also on $I := \bigcap_{i=1}^r \mm_{g_i}^{(\overline{H})}H$, for some $n,r \geq 1$ and distinct $g_1,\ldots,g_r\in G_H/\sim^{(\overline{H})}$. We claim that
\begin{equation} \label{done} A^+H \subseteq I + {A^+}^nH. 
\end{equation}
If $g_i\neq 1_{G_H}$ for all $i = 1, \ldots , r$, then $I$ and ${A^+}^nH$ are comaximal by the discussion in the proof of (2), so that (\ref{done}) is clear. Suppose on the other hand that, say, $g_1 = 1_{G_H}$.  Then $I + {A^+}^nH \subseteq A^+ H$, and the left $A$-module $A^+H/(I + {A^+}^n H)$ is annihilated by both of the comaximal ideals ${A^+}^{n-1}$ and by $\bigcap_{i=2}^r\mm_{g_i}^{(\overline{H})}$. Hence $A^+H/(I + {A^+}^n H) = \{0\}$ and again (\ref{done}) is proved. Thus (\ref{done}) holds in all cases and so $f \in \overline{H}^{\ast}$, as required.

\medskip

\noindent (8) Assume that $A$ is a domain. Localise $A$ at the maximal ideal $\mm_g$, and consider the left $A_{\mm_g}$-module $H_{\mm_g} := A_{\mm_g} \otimes_A H$. Since $H$ is a finitely generated projective $A$-module by Theorem \ref{properties}(6), and $A_{\mm_g}$ is local, $H_{\mm_g} $ is a free $A_{\mm_g}$-module of finite rank $r$. Let $Q$ denote the quotient field of $A$. Since $H$ is a torsion free $A$-module, 
$$ r \; = \; \mathrm{dim}_Q (Q \otimes_{A_{\mm_g}} H_{\mm_g}) \; = \; \mathrm{dim}_Q (Q \otimes_A H),$$
so $r$ is constant as $g$ varies through $G_H$. Taking $g = 1_{G_H}$ and applying Nakayama's lemma, 
\begin{equation}\label{rank} r = \mathrm{dim}_k (H/\mm_{1_{G_H}}H) = \mathrm{dim}_k (\overline{H}),
\end{equation}
proving the second equality in (\ref{size}). The first equality is a consequence of the second, using the antipode. Since the algebra $H/\mm_{g}^{(\overline{H})}H$ is artinian, it is isomorphic as an $A/\mm_g^{(\overline{H})}$-module to the factor $H_{\mm_g}/\mm_{g}^{(\overline{H})}H_{\mm_g}$ of $H_{\mm_g}$. But the latter is a free $A/\mm_{g}^{(\overline{H})}$-module of rank $r$. Combining this with (\ref{rank}) proves the equality in (\ref{hatsize}), and the inequality is then given by \cite[Theorem 6.1(2)]{BrCou}. If $H$ is prime then $A$ is a domain and the inequality (\ref{extra}) follows from \cite[Corollary 6.2 and its proof]{BrCou}. 

\end{proof}

\subsection{Subalgebras and subcoalgebras of $\widehat{kG_H}$}\label{widealg} In order to better control the structure of $\widehat{kG_H}$ we need to impose on the inclusion $A \subseteq H$ one or both of the additional hypotheses $\mathbf{(CoSplit)}$ and $\mathbf{(OrbSemi)}$ defined respectively in $\S$\ref{subsect1.1} and Definition \ref{orbsemi}. First we assume $\mathbf{(CoSplit)}$ in Proposition \ref{subgroup}, then $\mathbf{(OrbSemi)}$ in Proposition \ref{kGsubalg}, before combining these in Theorem \ref{decompkG}.

\begin{Prop}\label{subgroup} Let $A \subseteq H$ be an affine commutative-by-finite Hopf $k$-algebra satisfying $\mathbf{(CoSplit)}$. Assume also that $A$ is semiprime or central, or that $H$ is pointed. Keep the notation introduced in $\S\S$\ref{charcomp} and \ref{stand} and in Lemma \ref{HbarHopf}. Define
$$\widetilde{G_H} = \{ \Pi^{\circ} (\chi_g) : g \in G_H\} \subseteq H^{\circ}.$$ 

\begin{enumerate}
\item $\widetilde{G_H}$ is a subgroup of the group of units of $H^{\circ}$, and the algebra injection $\Pi^{\circ}: A^{\circ} \to H^{\circ}$ restricts to an isomorphism of groups from the group-likes $G_H$ of $A^{\circ}$ to $\widetilde{G_H}$.
\item The subalgebra of $H^{\circ}$ generated by $\widetilde{G_H}$ is the group algebra $k\widetilde{G_H}$, and $\Pi^{\circ}$ restricts to an algebra isomorphism from $kG_H$ to $k\widetilde{G_H}$.
\item For all $g \in G_H$, $\overline{H}^{\ast}\Pi^{\circ}(\chi_g) = \Pi^{\circ}(\chi_g)\overline{H}^{\ast}$ is free of rank 1 as an $\overline{H}^{\ast}$-bimodule. Hence $\langle \overline{H}^{\ast}, \widetilde{G_H} \rangle$ is a skew group algebra $\overline{H}^{\ast} \# \widetilde{G_H}$; that is,
$$ \langle \overline{H}^{\ast}, \widetilde{G_H} \rangle \; = \; \overline{H}^{\ast} \# \widetilde{G_H} \; = \; \oplus_{g \in G_H} \overline{H}^{\ast}\Pi^{\circ}(\chi_g) $$
is a subalgebra of $H^{\circ}$ contained in the subcoalgebra $\widehat{kG_H}$.

\item  There are inclusions 

$$ \widehat{kG_H} = \bigoplus_{g \in G_H/\sim^{(\overline{H})}} \widehat{g} \supseteq \bigoplus_{g \in G_H/\sim^{(\overline{H})}}\left(\bigoplus_{h\sim^{(\overline{H})} g} (H/H\mathfrak{m}_h)^{\ast}\right) \supseteq \bigoplus_{g \in G_H/\sim^{(\overline{H})}}\left(\bigoplus_{h\sim^{(\overline{H})} g} \overline{H}^{\ast}\Pi^{\circ}(\chi_h)\right) = \overline{H}^{\ast} \# \widetilde{G_H}. $$
If $A$ is a domain then the second inclusion is an equality.
\end{enumerate}
\end{Prop}

\begin{proof}(1),(2) The $\mathbf{(CoSplit)}$ hypothesis ensures that $\Pi$ is a coalgebra map, so $\Pi^\circ$ is an algebra homomorphism, and is an injection by Lemma \ref{HbarHopf}(5). Thus $\Pi^{\circ}$ maps $G_H$ isomorphically to a subgroup $\widetilde{G_H}$ of the group of units of $H^{\circ}$, and $$ k\widetilde{G_H} = \Pi^{\circ}(kG_H) \cong kG_H.$$

(3) By Remark \ref{decomprmks}(2), the action of $kG_H\subset A^\circ$ on $\overline{H}^*$ is 
$$ \chi_g\cdot f = \Pi^\circ(\chi_g) f \Pi^\circ(\chi_{g^{-1}}) \in \overline{H}^*, $$ 
for any $f\in \overline{H}^*, g\in G_H$, hence $\Pi^\circ(\chi_g)f\in \overline{H}^*\Pi^\circ(\chi_g)$, proving $\Pi^\circ(\chi_g)\overline{H}^* \subseteq \overline{H}^*\Pi^\circ(\chi_g)$ and the converse is proved similarly. Both these $\overline{H}^{\ast}$-modules are free of rank 1 since $\Pi^{\circ}(G_H)$ consists of units of $H^{\circ}$. From Theorem \ref{decompdual}(1) and its proof we know that $H^{\circ}$ is a free left and right $\overline{H}^{\ast}$-module with any vector space basis of $\Pi^{\circ}(A^{\circ})$ as free basis. In particular, $\{\Pi^{\circ}(\chi_g) : g \in G_H \}$ spans a free left and right $\overline{H}^{\circ}$-submodule of $\widehat{kG_H}$.

(4) The first equality and first inclusion are immediate from Definition \ref{charcompdef} and Lemma \ref{kGprop}(5). The second inclusion follows from Lemma \ref{kGprop}(5) and the third equality is part (3) of the present lemma. If $A$ is a domain the second inclusion is seen to be an equality by comparing vector space dimensions using Lemma \ref{kGprop}(8). 

\end{proof}

We examine next the effect of the $\mathbf{(OrbSemi)}$ hypothesis on $\widehat{kG_H}$.
\begin{Prop}\label{kGsubalg}
Let $A \subseteq H$ be an affine commutative-by-finite Hopf algebra satisfying $\mathbf{(OrbSemi)}$. 
\begin{enumerate}
\item The subcoalgebra $\widehat{kG_H}$ of $H^{\circ}$ is a Hopf subalgebra.
\item For $g,h \in G_H$, 
\begin{equation}\label{productghat}
\widehat{g}\widehat{h}\subseteq \sum_{g'\sim^{({\overline{H}})} g, h'\sim^{({\overline{H}})} h} \widehat{g'h'}.
\end{equation}
\item  For $g \in G_H$, $\widehat{g}$ is a left and right $\overline{H}^{\ast}$-module.
\end{enumerate}
Assume further that $A$ is a domain. Then, for any $g\in G_H$:
\begin{enumerate}
\setcounter{enumi}{3}

\item $\mathrm{dim}_k (\widehat{g}) = \left|\OO_{\mm_g}\right| \, \mathrm{dim}_k (\overline{H}).$

\item $$ \widehat{g} = \bigoplus_{h \sim^{({\overline{H}})} g} \left( H/H\mm_h \right)^{\ast} = \bigoplus_{h \sim^{({\overline{H}})} g}\left( H/ \mm_hH  \right)^{\ast}. $$
\end{enumerate}
\end{Prop}
\begin{proof}
Note that (3) is a special case of (2), since $\overline{H}^{\ast} = \widehat{1_{G_H}}$. By Lemma \ref{kGprop}(3),(6), (1) will follow if $\widehat{kG_H}$ is closed under multiplication. Moreover, by Lemma \ref{kGprop}(2), closure under multiplication will follow at once from (2). We claim that to prove (\ref{productghat}) amounts to showing that
\begin{equation}\label{coproductmg}
\Delta\left( \bigcap_{g'\sim^{(\overline{H})} g, h'\sim^{(\overline{H})} h} \mm_{g'h'} \right)\subseteq \mm_g^{(\overline{H})}\otimes A + A\otimes \mm_h^{(\overline{H})} =: J.
\end{equation}
For, the flatness of $H$ over $A$ together with the Chinese remainder theorem implies that $\bigcap_{g',h'} (H\mm_{g'h'}^{(\overline{H})}) = H\left(\bigcap_{g',h'} \mm_{g'h'}^{(\overline{H})}\right)$. Hence, given (\ref{coproductmg}),
\begin{equation}\label{upper} \Delta\left(\bigcap_{g'\sim^{({\overline{H}})} g, h'\sim^{({\overline{H}})} h} H\mm_{g'h'}^{(\overline{H})}\right) \subseteq H\mm_g^{(\overline{H})}\otimes H + H \otimes H\mm_h^{(\overline{H})}.
\end{equation}
Consider now $\beta\gamma $ for $\beta \in\widehat{g}$ and $\gamma \in\widehat{h}$.  It follows from (\ref{upper}) that
$$ \beta \gamma \left( \bigcap_{g'\sim^{({\overline{H}})} g, h'\sim^{({\overline{H}})} h} H\mm_{g'h'}^{(\overline{H})} \right) \subseteq \beta (H\mm_g^{(\overline{H})})\gamma (H) + \beta (H)\gamma (H\mm_h^{(\overline{H})}) = 0,$$
proving part (2) of the proposition.

Let us prove (\ref{coproductmg}). By the definition of the coproduct of $A$,
\begin{eqnarray*} \Delta\left( \bigcap_{g'\sim^{({\overline{H}})} g, h'\sim^{({\overline{H}})} h} \mm_{g'h'} \right) &\subseteq& \bigcap_{g'\sim^{({\overline{H}})} g, h'\sim^{({\overline{H}})} h} \Delta\left( \mm_{g'h'}\right)\\
 &\subseteq& \bigcap_{g'\sim^{({\overline{H}})} g, h'\sim^{({\overline{H}})} h} (\mm_{g'}\otimes A + A\otimes\mm_{h'}) =: I. 
\end{eqnarray*} 
Thus  (\ref{coproductmg}) will follow if it can be shown that $I = J$. It is clear that $J \subseteq I$. To prove equality, let $r=|\OO_{\mm_g}|$ and $s=|\OO_{\mm_h}|$. Then it follows from $\mathbf{(OrbSemi)}$ that
$$ \mathrm{dim}_k((A \otimes A)/J) = rs = \mathrm{dim}_k((A \otimes A)/I) .$$
This completes the proof of (\ref{coproductmg}).

\medskip

(4) Since $H$ is orbitally semisimple, $\left| \OO_{\mm_g} \right| = \mathrm{dim}_k (A/ \mm_g^{(\overline{H})})$ by (\ref{equisemi}), hence (4) is (\ref{hatsize}).

\medskip

(5) Since $H$ is orbitally semisimple, $H/H\mm_g^{(\overline{H})} = \bigoplus_{h \sim^{(\overline{H})} g} H/H\mm_h$, so this follows from the definition of $\widehat{g}$.
\end{proof}

\begin{Rmk} The formula (\ref{productghat}) cannot be further simplified. To see this, consider the group algebra $H := kD$ of the dihedral group as in Remark \ref{decomprmks}(1). 
Recall that $A = \; k[b^{\pm 1}] \; = \; \mathcal{O}(k^{\times}),$ so, $G_H = k^{\times}$. Now $\overline{H} = k\langle a \rangle$, so the $\overline{H}$-orbits  in $\mathrm{Maxspec}(A)$ are 
\begin{equation}\label{huh}   \{ \mathfrak{m}_g, \mathfrak{m}_{g^{-1}} \, : \, g \in k^{\times} \}, 
\end{equation}
and $\overline{H}^{\ast}  \cong  \overline{H}.$ Clearly $H$ satisfies both $\mathbf{(CoSplit)}$ with $X := (a-1)kA$, and $\mathbf{(OrbSemi)}$, since $\overline{H}$ is a group algebra. Therefore, by Propositions \ref{subgroup}(4) and \ref{kGsubalg}(5), and using (\ref{huh}),
\begin{equation}\label{sink} \widehat{g} \; = \;  (k\Pi^\circ(\chi_g) + k\Pi^\circ(\chi_{g^{-1}}))\otimes kC_2,
\end{equation}
and $\widehat{1}=1\otimes kC_2$ and $\widehat{-1}=k\Pi^\circ(\chi_{-1})\otimes kC_2$. Hence, for $g,h\in k^\times\setminus\lbrace \pm 1 \rbrace$, 
$$ \widehat{g}\widehat{h} = \left[ k\Pi^\circ(\chi_{gh}) + k\Pi^\circ(\chi_{g^{-1}h}) + k\Pi^\circ(\chi_{gh^{-1}}) + k\Pi^\circ(\chi_{g^{-1}h^{-1}}) \right]\otimes kC_2 = \widehat{gh}+\widehat{g^{-1}h}. $$ 
Thus the inclusion of (\ref{productghat}) is an equality in this instance.
\end{Rmk}

\medskip

\subsection{Decomposition of $\widehat{kG_H}$}

Before uniting the conclusions of Propositions \ref{subgroup} and \ref{kGsubalg} in Theorem \ref{decompkG} we need one more lemma. Recall from \textsection \ref{stand} and Lemma \ref{HbarHopf} the Hopf surjection $\iota^{\circ} : H^{\circ} \longrightarrow A^{\circ}$, so that $H^{\circ}$ is a right $A^{\circ}$-comodule algebra via the map $\rho : = (\id \otimes \iota^{\circ})\circ \Delta$.

\begin{Lem}\label{decompkGlem}
Let $A \subseteq H$ be an affine commutative-by-finite Hopf algebra satisfying $\mathbf{(OrbSemi)}$. Assume that $A$ is semiprime or central, or that $H$ is pointed. In the notation of \textsection\textsection \ref{stand}, \ref{orbitsec}, Definition \ref{charcompdef} and Lemma \ref{HbarHopf}, the following hold for each $g \in G_H$:
\begin{enumerate}
\item $\qquad \iota^\circ(\widehat{g}) = \bigoplus_{h\sim^{({\overline{H}})}  g} k\chi_h$.

\item $\qquad \widehat{g} \; \supseteq \; \lbrace f\in H^\circ: \rho(f) \in f\otimes \sum_{h\sim^{({\overline{H}})} g} k \chi_h \rbrace$.
\item $\qquad \widehat{1_{G_H}} \; = \; \overline{H}^{\ast}\; = \; \lbrace f\in H^\circ: \rho(f)=f\otimes 1 \rbrace$.
\end{enumerate}
\end{Lem}
\begin{proof}

(1) Let $g \in G_H$ and take $f\in\widehat{g}$, so that $\left.f\right|_A(\mm_g^{(\overline{H})}) =0$. Since $A \subseteq H$ satisfies $\mathbf{(OrbSemi)}$ the Chinese remainder theorem implies that
$$ \iota^{\circ}(f) = \left.f\right|_A\in \left(A/\mm_g^{(\overline{H})}\right)^{\ast}\cong \bigoplus_{h\sim^{({\overline{H}})} g}(A/\mm_h)^{\ast} = \bigoplus_{h\sim^{({\overline{H}})} g} k\chi_h. $$ 
Conversely, let $\alpha=\sum_{h\sim^{({\overline{H}})} g} \lambda_h\chi_h \in A^{\circ}$ for some $\lambda_h\in k$. Then $\alpha=\iota^\circ(\Pi^\circ(\alpha))$ by Lemma \ref{HbarHopf}(5). Recall that $H=A\oplus X$ as right $A$-modules by (\ref{split}) in $\S$\ref{subsect1.2}. Thus 
$$ \Pi^\circ(\alpha)(H\mm_g^{(\overline{H})}) = \Pi^\circ(\alpha)(\mm_g^{(\overline{H})}\oplus X\mm_g^{(\overline{H})}))= \alpha(\mm_g^{(\overline{H})}) = \sum_{h\sim^{(\overline{H})} g} \lambda_h\chi_h(\mm_g^{(\overline{H})})=0, $$ 
as $\mm_g^{(\overline{H})}\subseteq \mm_h$ for every $h\sim^{(\overline{H})} g$. Thus $\Pi^\circ(\alpha)\in\widehat{g}$ and $\alpha\in\iota^\circ(\widehat{g})$, proving (1).

\medskip

\noindent (2) Let $f \in H^{\circ}$ and suppose $\rho(f)=f\otimes \sum_{h\sim^{(\overline{H})} g} \lambda_h \chi_h$ for some $\lambda_h\in k$. Then
\begin{eqnarray*}
f(H\mm_g^{(\overline{H})}) &=& \sum f_1(H) f_2(\mm_g^{(\overline{H})}) = \sum f_1(H) \left.f_2\right|_A (\mm_g^{(\overline{H})}) \\
&=& \rho(f)(H\otimes \mm_g^{(\overline{H})}) = f(H)\left( \sum_{h\sim^{({\overline{H}})} g} \lambda_h \chi_h(\mm_g^{(\overline{H})}) \right) =0,
\end{eqnarray*}
as $\mm_g^{(\overline{H})}\subseteq \mm_h$ for every $h\sim^{(\overline{H})} g$. Hence $f\in \widehat{g}$.

\noindent (3) That the left side contains the right is a special case of (2). For the reverse inclusion, let $f\in\overline{H}^*=\widehat{1_G}$. By part (1) 
$$\rho(f)\in (\id \otimes \iota^{\circ})(\overline{H}^* \otimes \overline{H}^*) = \overline{H}^*\otimes k \varepsilon_A, $$ 
say $\rho(f)=f'\otimes \varepsilon_A$ for some $f'\in\overline{H}^*$. By the counit axiom of right comodules, $f=\varepsilon_A(1_A)f'=f'$, hence $\rho(f)=f\otimes 1_{A^\circ}$, proving (3).
\end{proof}

We can now show that when $A \subseteq H$ satisfies $\mathbf{(OrbSemi)}$ $\widehat{kG_H}$ decomposes as a crossed or smash product whenever $H^{\circ}$ does, generalising the presence of the group algebra $kG_H$ in the dual of a commutative semiprime Hopf algebra. 

\begin{Thm}\label{decompkG}
Let $A \subseteq H$ be an affine commutative-by-finite Hopf algebra satisfying $\mathbf{(OrbSemi)}$, and assume that $A$ is semiprime or central, or that $H$ is pointed. 
\begin{enumerate}
\item Suppose that $H^{\circ}\cong  \overline{H}^{\ast}\#_\sigma A^{\circ}$ decomposes as a crossed product for some cocycle $\sigma$ and action of $A^{\circ}$ on $\overline{H}^*$. Then
\begin{equation}\label{kGcrossed} \widehat{kG_H}\; \cong \; \overline{H}^* \#_{\left.\sigma\right|_{kG_H \otimes kG_H}} kG_H 
\end{equation}
as algebras, left $\overline{H}^{\ast}$-modules and right $kG_H$-comodules.

\item Suppose $A \subseteq H$ satisfies $\mathbf{(CoSplit)}$. Then, in (\ref{kGcrossed}), $\sigma$ is trivial, so  $\widehat{kG_H}$ is a smash product of $\overline{H}^{\ast}$ by $\widetilde{G_H}$, 
\begin{equation}\label{kGsmashed} \widehat{kG_H} \; = \; \overline{H}^* \# k\widetilde{G_H},
\end{equation} 
with action given, for $g\in G_H, \varphi\in\overline{H}^*$, by
$$ \Pi^{\circ}(\chi_g)\cdot \varphi = \Pi^\circ(\chi_g)\varphi\Pi^\circ(\chi_{g^{-1}}). $$  
\end{enumerate}

\end{Thm}
\begin{proof}
(1) Suppose that  $H^\circ\cong \overline{H}^* \#_\sigma A^\circ$. By \cite{DT} (see Remark \ref{propremarks}(4)), there exists a convolution invertible right $A^\circ$-comodule map $\gamma:A^\circ\to H^\circ$ such that the isomorphism is given by $f\# \varphi\mapsto f\gamma(\varphi)$. The core of this proof is showing that 
\begin{equation}\label{core} \widehat{g}=\bigoplus_{h\sim^{(\overline{H})} g} \overline{H}^*\gamma(\chi_h). \end{equation}

We first claim that $\gamma(\chi_h)\in\widehat{g}$. Since $\gamma:A^\circ\to H^\circ$ is a right $A^\circ$-comodule map and $\chi_h$ is a grouplike element of $A^{\circ}$, then 

$$ \rho\gamma(\chi_h) = (\gamma\otimes \id)\Delta_{A}(\chi_h) = \gamma(\chi_h)\otimes \chi_h $$
and the claim follows from Lemma \ref{decompkGlem}(2). Therefore, by Proposition \ref{kGsubalg}(3) 
$$ \overline{H}^*\gamma(\chi_h) \subseteq \overline{H}^*\widehat{g}\subseteq \widehat{g}, $$ 
proving the inclusion of the right of (\ref{core}) in the left.

For the reverse inclusion, let $f\in\widehat{g}$. By the second sentence of the proof, $f=\sum_{i=1}^n f_i\gamma(u_i\#\chi_{g_i})$, where $f_i\in\overline{H}^*, u_i\in U(\mathfrak{g}), g_i\in G_H$. Since $\rho$ is an algebra homomorphism and $\gamma:A^\circ\to H^\circ$ is a right $A^\circ$-comodule map, and since $\chi_{g_i} \in A^{\circ}$ is grouplike and $\rho(f_i)=f_i\otimes 1$ by Lemma \ref{decompkGlem}(2) we have
\begin{eqnarray*}
\rho(f) &=& \sum_{i=1}^n \rho(f_i)\rho\gamma(u_i\#\chi_{g_i}) = \sum_{i=1}^n \rho(f_i) (\gamma\otimes \id)\Delta(u_i\#\chi_{g_i}) \\
&=& \sum_{i=1}^n \sum_{(u_i)} f_i\gamma(u_{i,1}\#\chi_{g_i})\otimes (u_{i,2}\#\chi_{g_i}).
\end{eqnarray*}
But by Lemma \ref{decompkGlem}(1) $$ \rho(f)\in \widehat{g}\otimes\iota^\circ(\widehat{g})=\bigoplus_{h\sim^{(\overline{H})} g} \widehat{g}\otimes \chi_h.$$ 
Moreover, since each $u_i$ is a polynomial on primitive elements, we must have $g_i\sim^{(\overline{H})} g$ for every $i$, and each $u_i=1$. Hence $f$ has the required form, and (\ref{core}) follows.

Therefore, $$ \widehat{kG_H} = \overline{H}^*\gamma(kG_H) \cong \overline{H}^*\otimes kG_H $$ as vector spaces, left $\overline{H}^*$-modules and right $kG_H$-comodules, since by the calculations above $\rho(\widehat{kG_H})\subseteq \widehat{kG_H}\otimes kG_H$, that is the $A^\circ$-comodule structure of $H^\circ$ restricts to a $kG_H$-comodule structure for $\widehat{kG_H}$. In particular, by Lemma \ref{HbarHopf} $$ \widehat{kG_H}^{co\, kG_H} = \widehat{kG_H}\cap (H^\circ)^{co\,A^\circ} = \widehat{kG_H}\cap\overline{H}^* = \overline{H}^*. $$ Also, since $\gamma: A^\circ\to H^\circ$ is convolution invertible, then so is its restriction to the coradical $\left.\gamma\right|_{kG_H}:kG_H\to \widehat{kG_H}$, \cite[Lemma 14]{Tak}. Therefore, $$ \widehat{kG_H} \cong \overline{H}^* \#_\tau kG_H $$ and by inspecting the formulae for the cocyles $\sigma$ and $\tau$ one easily sees that $\tau=\left.\sigma\right|_{kG_H\otimes kG_H}$.

\medskip

\noindent (2) If $X$ can be chosen to be a coideal, then $H^\circ\cong \overline{H}^* \# A^\circ$ by Theorem \ref{decompdual}(1), hence the cocycle $\sigma$ is trivial and $\widehat{kG_H}$ decomposes as a smash product. The action of $kG_H$ on $\overline{H}^*$ is a special case of Remark \ref{decomprmks}(2).

\end{proof}

\medskip

\section{$H^{\circ}$ under the auspices of $\mathbf{(CoSplit)}$ and $\mathbf{(OrbSemi)}$}\label{final}

Let $A \subseteq H$ be an affine commutative-by-finite Hopf algebra. Recall that the subspaces $W(H^{\circ})$ and $\widehat{kG_H}$ of $H^\circ$ were introduced in Definition\ref{Wdef} and \ref{charcompdef}. In the presence of $\mathbf{(CoSplit)}$ and $\mathbf{(OrbSemi)}$ one has for $H^{\circ}$ the following generalisation of the description of the Hopf dual of a commutative Hopf algebra via the Cartier-Gabriel-Kostant theorem.
\begin{Thm}\label{crux}
Let $A \subseteq H$ be an affine commutative-by-finite Hopf algebra and assume that $A$ is semiprime or central, or that $H$ is pointed. Suppose that $A \subseteq H$ satisfies 
$\mathbf{(CoSplit)}$ and $\mathbf{(OrbSemi)}$.  Retain the notation of $\S$\ref{stand}, $\S$\ref{commcase} and of Definitions  \ref{Wdef} and \ref{charcompdef}.   

\begin{enumerate}
\item $H^\circ$ is a smash product
\begin{equation}\label{finalHdual}
H^\circ \;\cong \; \overline{H}^\ast \#\, A^\circ \; \cong \; (\overline{H}^\ast \# \,A')\# kG_H,
\end{equation}
where $A/N(A) = \mathcal{O}(G_H)$, where $N(A)$ is the nilradical of $A$. If $\mathrm{char}k = 0$ then $A' = U(\gg)$, where $\gg$ is the Lie algebra of $G_H$.

\item $H^\circ$ contains three Hopf subalgebras: $\overline{H}^\ast$,
\begin{equation}\label{finalWdecomp}
W(H^\circ) \; \cong \; \overline{H}^\ast \#\, A'
\end{equation}
and
\begin{equation}\label{finalkGdecomp}
\widehat{kG_H}\; \cong \; \overline{H}^\ast \#\, kG_H,
\end{equation}
with $A' \cong U(\gg)$ if $k$ has characteristic 0.
\item Suppose that $\overline{H}$ is semisimple and $k$ has characteristic 0. Then the action of $U(\gg)$ on $\overline{H}$ is inner, so that
\begin{equation}\label{coup} H^{\circ} \; \cong \; (\overline{H}^{\ast} \otimes U(\gg)) \# kG_H,
\end{equation}
a skew group algebra with coefficient ring the Hopf subalgebra $W(H^{\circ})$.
\end{enumerate}
\end{Thm}
\begin{proof}(1),(2) By the isomorphism (\ref{comm}) in $\S$\ref{commcase} there is a decomposition
\begin{equation}\label{first}  A^{\circ} \; = \;  A' \# kG_H,
\end{equation}
with $A'$ and $kG_H$ Hopf subalgebras of $A^{\circ}$. Substituting (\ref{first}) into the isomorphism (\ref{smashed}) of Theorem \ref{decompdual}(1) shows that, as algebras,
\begin{equation}\label{second}  H^{\circ} \; \cong \; \overline{H}^{\ast}\# A^{\circ} \; = \; \overline{H}^{\ast} \# (A' \# kG_H).
\end{equation}
More precisely, the proof of Theorem \ref{decompdual}(1), using the fact that $\Pi^{\circ}$ is an algebra homomorphism, shows that we can rewrite (\ref{second}) as
\begin{equation}\label{third} H^{\circ} \; = \; \overline{H}^{\ast} \# \Pi^{\circ}(A' \# kG_H)\; = \; \overline{H}^{\circ} \# (\Pi^{\circ}(A') \# \Pi^{\circ}(kG_H)) \; = \; (\overline{H}^{\circ} \# \Pi^{\circ}(A')) \# \Pi^{\circ}(G_H).
\end{equation}
From (\ref{third}), by Theorem \ref{decompW}(2),(5) and their proofs, $H^{\circ}$ contains as a normal Hopf subalgebra the smash product
$$ \overline{H}^{\ast} \# \Pi^{\circ}(A') \; = \; W(H^{\circ}). $$
Moreover, by Proposition \ref{subgroup}(3), $H^{\circ} = \overline{H}^{\ast} \# \Pi^{\circ}(A' \# kG_H)$ contains the skew group algebra 
$$ \overline{H}^{\ast} \# \Pi^{\circ}(G_H) \; = \; \overline{H}^{\ast} \# \widetilde{G_H}, $$
and this equals the Hopf subalgebra $\widehat{kG_H}$ of $H^{\circ}$ thanks to the equality (\ref{kGsmashed}) of Theorem \ref{decompkG}.

(3) This is Corollary \ref{inner} combined with parts (1) and (2).
\end{proof}

\begin{Rmks}\label{cruxqns}(1) For the record, we state explicitly here the question about our standing hypotheses already mentioned in $\S$\ref{subsect1.3}.

\begin{Qtn}\label{hyps} Let $H$ be an affine commutative-by-finite Hopf algebra. Is there always a choice of normal commutative Hopf subalgebra $A$ of $H$ such that $A \subseteq H$ satisfies $\mathbf{(CoSplit)}$ and $\mathbf{(OrbSemi)}$?
\end{Qtn}

We review the status of this question for a large number of families of examples in $\S$\ref{Examples}. 

\medskip

\noindent (2) The smash product decompositions in Theorem \ref{crux} in general have nontrivial actions. For example, the actions in (\ref{finalWdecomp}) and (\ref{finalkGdecomp}) are nontrivial for $H=B(n,w,q)$ defined in \S\ref{GKdim1}, see Proposition \ref{GK1cases}(IV) and Remarks \ref{details}(2), with the references given there; and the skew group action in (\ref{finalHdual}) is in general nontrivial even when $H = A$ is commutative - for example, the algebra structure of $\mathcal{O}(G)^{\circ}$ when G is a semidirect product of $(k,+)$ by $k^{\times}$ is determined in \cite[Appendix A.4]{Cou}.
\end{Rmks}

\bigskip

\section{Examples}\label{Examples}

We list here a number of classes of examples of affine commutative-by-finite Hopf algebras and discuss in each case what our results tell us about their Hopf duals.

\subsection{Enveloping algebras of Lie algebras in positive characteristic}\label{envposchar} Assume that $k$ has positive characteristic $p$. Let $\gg$ be a $k$-Lie algebra with basis $\{x_1, \ldots , x_n\}$. Then $H := U(\gg)$ is a free module of finite rank over a central polynomial Hopf subalgebra 
$$A:=k[y_1,\ldots,y_n],$$
where each $y_i$ is a $p$-polynomial in $x_i$ and hence is primitive; see \cite[Proposition 1]{Jac}.  When $\gg$ is restricted, with restriction map $x\mapsto x^{[p]}$, one takes $y_i := x_i^p - x_i^{[p]}$ for all $i$, and in this case the Hopf quotient $\overline{H}$ is the restricted enveloping algebra of $\gg$, usually denoted by $u^{[p]}(\gg)$, with $\mathrm{dim}_k (u^{[p]}(\gg)) = p^n$. In the non-restricted case it is still the case by \cite[Proposition 2]{Jac} that $U(\gg)$ is a free $A$-module, with basis 
\begin{equation}\label{basis} \{x_1^{j_1} \cdots x_n^{j_n} : 0 \leq j_i < p^{d_i} \}, 
\end{equation}
where $y_i$ is a linear combination of $x_{i}^{p^j}$ for $j = 0, \ldots , d_i$. Thus, in the general non-restricted case,
$$ \mathrm{dim}_k(\overline{H}) \; = \; p^{\sum_i d_i}.$$

Since $A$ is central $\mathbf{(OrbSemi)}$ is trivially true for $H = U(\gg)$. Moreover it is clear from (\ref{basis}) that, as an $A$-module complement to $A$ in $H$, we can choose 
$$ X \; := \; \sum Ax_1^{j_1} \cdots x_n^{j_n}, \; 0 \leq j_i < p^{d_i}, \; \textit{ not all } j_i = 0.$$
Thus $X$ is a coideal of $H$ and so $\mathbf{(CoSplit)}$ is also valid. We can therefore apply Theorem \ref{crux}(1). Note that 
$$k[y_1, \ldots , y_n]\; \cong \; \mathcal{O}((k,+)^n),$$
and that $H^{\circ}$ is commutative since $H$ is cocommutative, so that the actions in the smash products composing $H^{\circ}$ are all trivial. We conclude that, as algebras,
$$ U(\gg)^{\circ} \; \cong \; A' \otimes \overline{H}^{\ast} \otimes k(k,+)^n, $$
with $\overline{H}^{\ast}$, $W(H^{\circ}) \, := \, A' \otimes \overline{H}^{\ast}$ and $\widehat{k(k,+)^n}\, := \, \overline{H}^{\ast} \otimes k(k,+)^n$ each being a Hopf subalgebra of $U(\gg)^{\circ}$. Here, as an algebra, $A'$ is a divided powers algebra in $n$ variables, 
$$ A' \; = \; k[f^{(m)}_1, \ldots , f^{(m)}_n : m \geq 0],$$ 
as defined and discussed at \cite[Examples 5.6.8 and 9.1.7]{Mont}. Modulo the Hopf ideal $(\overline{H}^{\ast})^+ A'$ of $W(H^{\circ})$ the coproducts of the elements of $A'$ take their classical cocommutative forms. Similarly, the elements of $(k,+)^n$ are group-like \emph{modulo}$(\overline{H}^{\ast})^+k(k,+)^n$.

\medskip

\subsection{Group algebras of finitely generated abelian-by-finite groups}\label{abfingroup} Let $G$ be a finitely generated group with an abelian normal subgroup $N$ of finite index. For convenience we may as well choose $N$ to be torsion free, and hence free abelian of rank $n$, say. With $H := kG$ and $A:= kN$, $\overline{H} = k(G/N)$. Clearly $H$ is a free $A$-module with basis a set of coset representatives of $N$ in $G$, and we can choose as a complement to $A$ in $H$ the free $A$-submodule $X$ with basis the non-trivial coset representatives. Thus $X$ is a subcoalgebra of $H$, so $\mathbf{(CoSplit)}$ holds, and $\mathbf{(OrbSemi)}$ holds since the adjoint action on $A$ is by the group $G/N$. Thus once again Theorem \ref{crux}(1) applies, and again the smash products occurring are trivial because $H^{\circ}$ is commutative. We deduce that, as algebras,
$$ kG^{\circ} \; \cong  \;  k[z_1, \ldots , z_n] \otimes \overline{H}^{\ast} \otimes k((k^{\times})^n), $$
with 
$$\overline{H}^{\ast} \cong k^{\oplus|G/N|}$$
and $\overline{H}^{\ast}$, $W(H^{\circ}) \,:= \, k[z_1, \ldots , z_n] \otimes \overline{H}^{\ast}$ and $\widehat{k((k^{\times})^n)} \, := \, \overline{H}^{\ast} \otimes k((k^{\times})^n)$ each being a Hopf subalgebra of $kG^{\circ}$. As in $\S$\ref{envposchar}, the elements $z_i$ are primitive \emph{modulo}$(\overline{H}^{\ast})^+k[z_1, \ldots , z_n]$ and the elements of $(k^{\times})^n$ are group-like \emph{modulo}$(\overline{H}^{\ast})^+k((k^{\times})^n)$.

\medskip

\subsection{Quantized coordinate rings at a root of unity}\label{quantG} In this subsection and the next we assume that $k = \mathbb{C}$ and that $\ell$ is an odd positive integer with $\ell \geq 3$, and prime to 3 if $\gg$ contains a factor of type $G_2$. The quantized coordinate ring $\OO_q(G)$ of a connected, simply connected, semisimple Lie group $G$ is a noetherian Hopf algebra, \cite[Sections 4.1 and 6.1]{ConLy}. If $q=\epsilon$ is a primitive $\ell$th root of unity, $\OO_\epsilon(G)$ is a finite module over a central Hopf subalgebra isomorphic to $\OO(G)$, \cite[Proposition 6.4]{ConLy}, \cite[Theorem III.7.2]{BrGoodbook}. Thus 
$$ A \; := \; \mathcal{O}(G) \; \subseteq \; \mathcal{O}_{\epsilon}(G) =: H $$
is an affine commutative-by-finite Hopf algebra which satisfies $\mathbf{(OrbSemi)}$ because of the centrality of $A$. In fact, the extension $\OO(G)\subseteq \OO_\epsilon(G)$ is cleft in the sense of Remark \ref{propremarks}(3), as is noted in \cite[Remark 2.18(b)]{AndruGar}), with a cleaving map which is a coalgebra map, as shown in the proof of \cite[Proposition 2.8(c)]{AndruGar}. Thus $\mathbf{(CoSplit)}$ is satisfied. The finite-dimensional Hopf quotient $\overline{H} :=\OO_\epsilon(G)/\OO(G)^+\OO_\epsilon(G)$ is the restricted quantized coordinate ring, sometimes denoted by $o_\epsilon(G)$; its vector space dimension is $\ell^{\mathrm{dim}(G)}$. Each of $u_{\epsilon}(\gg)$, (defined in $\S$\ref{quantg}), and $o_{\epsilon}(G)$ is the Hopf dual of the other by \cite[Theorem III.7.10]{BrGoodbook}. Thus Theorem \ref{decompdual}(2) applies, yielding the isomorphism of algebras, left $u_\epsilon(\gg)$-modules and right $U(\gg)\#kG$-comodules
\begin{equation}\label{function} \OO_\epsilon(G)^\circ \; \cong \; u_\epsilon(\gg)\otimes (U(\gg)\# kG),
\end{equation} 
where $\gg=\Lie G$ denotes the Lie algebra of $G$ and $u_{\epsilon}(\gg)$ is $\overline{U_{\epsilon}(\gg)}$, as defined in $\S$\ref{quantg}.  By Theorems \ref{crux}(1), \ref{decompW}(5) and \ref{decompkG}(2), $\OO_{\epsilon}(G)$ contains the Hopf subalgebras $u_\epsilon(\gg)$, 
$$W(H^{\circ}) \; = \; u_\epsilon(\gg)\otimes U(\gg)$$
and 
$$ \widehat{kG_H} \;= \; u_\epsilon(\gg) \otimes kG, $$
and the tensorand $U(\gg)\# kG$ on the right of (\ref{function}) is isomorphic to $\mathcal{O}(G)^{\circ}$ as an algebra.

\medskip

\subsection{Quantized enveloping algebras at a root of unity}\label{quantg} The quantized enveloping algebra $U_q(\gg)$ of a semisimple finite-dimensional Lie algebra $\gg$ is a noetherian Hopf algebra \cite[Section 9.1]{ConProc}, \cite[Section 9.1A]{ChPr}. When $q=\epsilon$ is a primitive $\ell$th root of unity $U_\epsilon(\gg)$ is a free module of rank $\ell^{\dim\gg}$ over a central Hopf domain $A$ \cite[Corollary and Theorem 19.1]{ConProc}, \cite[Theorem III.6.2]{BrGoodbook}. The Hopf subalgebra $A$ is the coordinate ring of  a certain solvable Poisson algebraic group $T$ with $\mathrm{dim}T = \mathrm{dim}\mathfrak{g}$; see \cite{ConProc} or \cite[III.6.5]{BrGoodbook}. The finite-dimensional Hopf quotient $\overline{H}=U_\epsilon(\gg)/A^+U_\epsilon(\gg)$ of $U_\epsilon(\gg)$ is the restricted quantized enveloping algebra, denoted by $u_\epsilon(\gg)$.

Since $A$ is central in $U_\epsilon(\gg)$ $\mathbf{(OrbSemi)}$ holds, but we do not know whether $\mathbf{(CoSplit)}$ is valid. However, for the two easiest cases - that is when $\gg$ is $\mathfrak{sl}_2$ or $\mathfrak{sl}_3$ - $\mathbf{(CoSplit)}$ has been confirmed by hand, as follows. 

For $U_\epsilon(\mathfrak{sl}_2)$ with its standard generators $\{E,F, K^{\pm 1}\}$, $A = k[E^{\ell}, F^{\ell}, K^{\pm \ell}]$ and one may check that the $A$-submodule $X$ of  $U_\epsilon(\mathfrak{sl}_2(k))$ generated by
$$ \lbrace F^rK^sE^t: 0\leq r,s,t<l \text{ and } r,t \text{ not both zero} \rbrace \cup \lbrace K^s-1 : 1\leq s<l \rbrace $$ 
is a coideal of $U_\epsilon(\sl_2(k))$. When $\gg$ is $\mathfrak{sl}_2$, the group $T$ is a semidirect product, $T = (k,+)^2\rtimes k^{\times}$; see e.g. \cite[$\S$III.6.5]{BrGoodbook}. Let $\mathfrak{t}$ denote its Lie algebra. Since $u_\epsilon(\mathfrak{sl}_2)^{\ast} = o_\epsilon(SL(2))$ by \cite[Theorem III.7.10]{BrGoodbook}, we deduce from Theorem \ref{crux} that 
$$ U_\epsilon(\sl_2)^\circ \; \cong \; (o_\epsilon(SL_2)\#U(\mathfrak{t}))\#kT ,$$ 
with Hopf subalgebras $o_\epsilon(SL_2)$, $W(U_\epsilon(\sl_2)^{\circ})\,  := \, o_\epsilon(SL_2)\#U(\mathfrak{t})$ and $\widehat{kT} \, := \, o_\epsilon(SL_2)\#kT$.

\medskip

For $ U_\epsilon(\mathfrak{sl}_3)$, the Hopf centre $A$ is $\mathcal{O}(S)$ for a Poisson algebraic group $S$ which, using \cite[III.6.5]{BrGoodbook}, is a semidirect product $(k,+)^6 \rtimes (k^{\times})^2$, whose Lie algebra $\mathfrak{s}$ has basis $\lbrace e_1,e_2,e_3,f_1,f_2,f_3,k_1,k_2 \rbrace,$ 
with non-zero brackets 
$$ [e_1,e_2]=e_3, [f_1,f_2]=-f_3, [e_1,k_1]=-e_1, [e_2,k_2]=-e_2, $$ $$ [e_3,k_i]=-e_3, [f_1,k_1]=-f_1, [f_2,k_2]=-f_2, [f_3,k_i]=-f_3. $$ 
An explicit coalgebra $A$-module complement to $A$ in $U_\epsilon(\mathfrak{sl}_3)$ has been described as a result of extensive calculations in \cite[page 131 and Appendix A.1]{Cou}, so that $A \subseteq U_\epsilon(\mathfrak{sl}_3)$ satisfies $\mathbf{(CoSplit)}$. Therefore, applying Theorem \ref{crux}, we obtain
$$ U_\epsilon(\sl_3)^\circ \; \cong\;  (o_\epsilon(SL_3)\#U(\mathfrak{s}))\# kS,$$
with Hopf subalgebras $o_\epsilon(SL_3)$,
$$ W(U_\epsilon(\mathfrak{sl}_3)^{\circ}) \; \cong  \; o_\epsilon(SL_3)\#U(\mathfrak{s})$$
and 
$$ \widehat{kS} \; \cong \; o_\epsilon(SL_3)\# kS.$$

In view of these two cases it seems reasonable to ask

\begin{question}\label{qeacosplit} Does $U_\epsilon(\gg)$ satisfy $\mathbf{(CoSplit)}$ for all semisimple Lie algebras $\gg$?
\end{question}

\medskip

\subsection{Prime regular affine Hopf algebras of Gelfand-Kirillov dimension 1}\label{GKdim1} Here, ``regular'' means ``having finite global dimension,'' necessarily then equal to 1. These Hopf $k$-algebras were classified when $k$ is an algebraically closed field of characteristic 0 by Wu, Liu and Ding in \cite{Wu}, building on \cite{BrZh} and \cite{LWZ}; see also the survey article \cite{BrZh2}. By a fundamental result of Small, Stafford and Warfield \cite{SSW} a semiprime affine algebra of GK-dimension one is a finite module over its centre. But in fact more is true for these Hopf algebras - they are all commutative-by-finite. This can be checked  on a case-by-case basis, as we now briefly outline. For $k$ algebraically closed of characteristic 0 there are 2 finite families and 3 infinite families, as follows.

\begin{enumerate}
\item[(I)] The commutative algebras $k[x]$ and $k[x^{\pm 1}]$.
\item[(II)] The unique noncommutative cocommutative example, the group algebra $H = kD$ of the infinite dihedral group $D$, defined and discussed in Remarks \ref{decomprmks}(1). 
\item[(III)] The infinite dimensional Taft algebras 
$$T(n,t,q)\; :=\; k\langle g,x: g^n=1, xg=qgx \rangle,$$
 where $n \in \mathbb{Z}_{\geq 2},$ $1 \leq t \leq n$ and $q$ is a primitive $n$th root of 1 in $k$, with $g$ group-like and $\Delta (x) = x \otimes1 + g^t \otimes x$. With $n' :=n/\gcd(n,t)$, $A :=k[x^{n'}]$ is a commutative normal Hopf subalgebra.
\item[(IV)] The generalised Liu algebras $B(n,w,q)$, where $n$ and $w$ are positive integers and $q$ is a primitive $n$th root of 1. Here,
$$ B(n,w,q) \; :=\; k\langle x^{\pm 1}, g^{\pm 1}, y\,:\, x \textit{ central},\, yg=qgy, \, g^n=x^w=1-y^n \rangle, $$
with $x$ and $g$ group-like and $\Delta(y) = y \otimes 1 + g\otimes y$. One can show that $A:=k[x^{\pm 1}]$ is a central Hopf subalgebra over which $B(n,w,q)$ is free of rank $n^2$.
\item[(V)] Let $m$ and $d$ be positive integers with $(1+m)d$ even, and let $q$ be a primitive $2m^{\mathrm{th}}$ root of 1 in $k$. The Hopf algebras $D(m,d,q)$ are defined in \cite[Section 4.1]{Wu}. $D(m,d,q)$ is finitely generated over the normal commutative Hopf subalgebra $A:=k[x^{\pm 1}]$, \cite[(4.7)]{Wu}.
\end{enumerate} 

The above algebras are all free over the listed normal commutative Hopf subalgebras. Families (I)-(IV) are pointed and decompose as crossed products 
$$ H\; \cong \; A\#_\sigma \overline{H};$$
 but $D(m,d,q)$ is not pointed, \cite[Proposition 4.9]{Wu}. Applying Theorems \ref{decompdual}, \ref{decompW} and \ref{decompkG} to these families yields the following information about their Hopf duals, (omitting the algebras in (I) and (II), which are already covered by $\S\S$\ref{commcase} and \ref{abfingroup}). In the proposition, for coprime integers $s$ and $m$ with $1 \leq s < m$ and a primitive $m^{\mathrm{th}}$ root of unity $q \in k$, $T_f(m,s,q)$ denotes the finite dimensional Taft algebra in these parameters; that is, $T_f(m,s,q)  := T(m,s,q)/\langle x^m \rangle$. For a positive integer $m$, $C_m$ denotes the cyclic group of order $m$.

For reasons of space we omit some details in the proofs for (III) and (IV), and omit the entire proof of (V). Full proofs can be found in \cite{Jahn} and \cite{Cou}; detailed references are listed in Remarks \ref{details}.

\begin{proposition}\label{GK1cases} Let $k$ be an algebraically closed field of characteristic 0, and let $H$ be an affine prime regular Hopf $k$-algebra of Gelfand-Kirillov dimension one, as listed in cases (III), (IV) and (V) above. Then $\mathbf{CoSplit)}$ and $\mathbf{OrbSemi)}$ hold in all cases, and the finite dual of $H$ takes the following forms.
\begin{enumerate}[label=(\roman*)]
\item[(III)] Let $d:=\mathrm{gcd}(n,t),\, n':=n/d,\, t':=t/d$. Then, as algebras,
$$T(n,t,q)^\circ \; \cong \; kC_d \otimes T_f(n',t',q^d)\otimes  k[f]\otimes k(k,+), $$
with Hopf subalgebras $kC_d$, $\overline{H}^{\ast} = kC_d \otimes T_f(n',t',q^d)$,
$$ W(H^{\circ}) \; = \; kC_d \otimes T_f(n',t',q^d)\otimes  k[f]$$
and
$$ \widehat{kG_H} \; = \;  kC_d \otimes T_f(n',t',q^d)\otimes  k(k,+).$$

\item[(IV)] When $H = B(n,w,q)$,
$$H^\circ \; = \; T_f(n,1,q) \# (k[f]\otimes k(k^\times)),$$

with Hopf subalgebras $\overline{H}^{\ast} =T_f(n,1,q)$, $W(H^{\circ}) = T_f(n,1,q) \# k[f]$ and 
$$\widehat{kG_H} \; = \; T_f(n,1,q) \# k(k,+).$$

\item[(V)] When $H = D(m,d,q)$,
$$ H^\circ \; = \; (kC_2\#_\sigma T_f(m,1,q^2)) \# (k[f]\otimes k(k^\times)),$$
with Hopf subalgebras $\overline{H}^{\ast} = kC_2 \#_{\sigma} T_f(m,1,q^2)$, $W(H^{\circ}) = \overline{H}^{\ast}\#k[f]$ and 
$$ \widehat{kG_H} \; = \; \overline{H}^{\ast} \# k(k^{\times}).$$ 

\end{enumerate}
\end{proposition}

\begin{proof}(III) Let $A = k[x^{n'}]$ as noted above. Thus 
$$\overline{H} \; := \; H/ x^{n'}H \; = \; k\langle \overline{x}, \overline{g}\, : \, \overline{g}^n = 1,\, \overline{x}^{n'} = 0, \, \overline{x}\overline{g} = q\overline{g}\overline{x} \rangle.$$
Thus, as a right $A$-module $H = A \oplus X$, where $X$ is the free right $A$-module on basis $\mathcal{P}$, where
$$\mathcal{P} \; := \; \lbrace g^ix^j : 0\leq i<n, 1\leq j<n' \rbrace \cup \lbrace g^i-1 : 1\leq i<n \rbrace.$$
One checks that $X$ is a coideal of $H$, so that $H$ satisfies $\mathbf{(CoSplit)}$. Moreover, it is clear that, for $ 0\leq i<n,\, 0 \leq j<n'$, the map 
$$ \gamma: \overline{H}\longrightarrow H: \overline{g}^i\overline{x}^j \mapsto g^i x^j $$
is a cleaving map of coalgebras, so that Theorem \ref{decompdual}(2) applies, yielding the isomorphism of algebras
\begin{equation}\label{one} H^{\circ} \; \cong \; \overline{H}^{\ast} \otimes A^{\circ}.
\end{equation}
It is now straightforward to calculate that, as an algebra, 
$$ \overline{H} \; \cong \;  T_f(n',t',q^d) \#_{\tau} kC_d, $$
where $\tau$ is a cocycle, while as coalgebra
$$ \overline{H} \; \cong \;  T_f(n',t',q^d) \otimes kC_d.$$
Since $\mathrm{gcd}(n',t') = 1$, $T_f(n',t',q^d)$ is self-dual by \cite[Exercise 7.4.3]{Radbook}, as also is $kC_d$, so we deduce that, as an algebra,
\begin{equation}\label{two} \overline{H}^{\ast} \; \cong \; kC_d \otimes T_f(n',t',q^d).
\end{equation}
Combining (\ref{one}), (\ref{two}) and the fact that $A^{\circ} = k[f] \otimes k(k,+)$, we deduce the algebra isomorphism in (III), with $\overline{H}^{\ast}$ a Hopf subalgebra by Lemma \ref{HbarHopf}(1).  Moreover $A \subseteq H$ also satisfies $\mathbf{(OrbSemi)}$ since the action of $\overline{H}$ on $A$ factors through the group algebra $k\langle\overline{g}\rangle$. Thus the claims regarding $W(H^{\circ})$ and $\widehat{kG_H}$ follow respectively from Theorems \ref{decompW}(5) and \ref{decompkG}(2).

\medskip
(IV) Let $H = B(n,w,q)$, so $H = \bigoplus_{0 \leq i,j < n}y^ig^jA$ is a free $A$-module, and
\begin{equation}\label{three} \overline{H} \; = \; H/A^+H \; \cong \; T_f(n,1,q).
\end{equation}
As a complement $X$ to $A$ in $H$ choose the $A$-module generated by 
$$ \lbrace g^iy^j : 0\leq i<n, 1\leq j<n \rbrace \cup \lbrace g^i-1: 1\leq i<n \rbrace.$$
One checks routinely that $X$ is a coideal of $H$. Therefore $\mathbf{(CoSplit)}$ holds, and Theorem \ref{decompdual}(1) implies that 
$$ H^{\circ} \; \cong \; \overline{H}^{\ast} \# A^{\circ} \; \cong \; T_f(n,1,q) \# (k[f] \otimes k(k^{\times})),$$
by  (\ref{three}) and the fact that $T_f(n,1,q)$ is self-dual by \cite[Exercise 7.3.4]{Radbook}.

Since $A$ is central $A \subseteq H$ satisfies $\mathbf{(OrbSemi)}$, so Theorem \ref{crux}(2) confirms the claims regarding $W(H^{\circ})$ and $\widehat{kG_H}$.

\medskip

(V) The fact that $D(m,d,q)$ satisfies $\mathbf{(CoSplit)}$ and $\mathbf{(OrbSemi)}$ is proved in \cite[$\S$2.2.5, pages 52-55 and Proposition 3.1.13]{Cou}, and the structure of its Hopf dual is described in \cite[Corollary 4.4.6(V) and Appendix A.2]{Cou}.
\end{proof}

\begin{Rmks}\label{details}(1) Precise formulae for the comultiplication for the duals of the Taft algebras (III) have been obtained - for full details see \cite[Corollary 4.4.6 and Appendix A.2]{Cou}.

(2) In (IV), the smash products in both $W(H^{\circ})$ and $\widehat{k(k^{\times})}$ are non-trivial, hence so is the smash product in $H^{\circ}$ itself. Precise formulae are given in \cite[Corollary 4.4.6(4) and Appendix A.2]{Cou}. Formulae for the coproduct are given at the same reference - in particular, neither $k[f]$ nor $k(k^{\times})$ is a Hopf subalgebra of $H^{\circ}$.

(3) For completeness, we note some gaps in the above analysis. For the duals of the Taft algebras (III), we only have precise formulae for the coproduct when $\mathrm{gcd}(n,t) = 1$; see \cite[Remark 6.12]{Jahn} for this case. For the duals $D(m,d,q)^{\circ}$ in (V) we do not have formulae for the actions of $k[f]$ and $k(k^{\times})$ on $\overline{H}^{\ast}$, nor do we know whether the cocycle $\sigma$ is non-trivial.  
\end{Rmks}

\medskip

\subsection{Noetherian PI Hopf domains of Gelfand-Kirillov dimension two}\label{GKdim2} Let $k$ be algebraically closed of characteristic 0 and let $H$ be a noetherian Hopf $k$-algebra domain with $\GKdim (H) = 2$. Such Hopf algebras were classified in \cite{GoodZh} under the additional assumption that $H$ has an infinite dimensional commutative factor, or equivalently by \cite[Proposition 3.8(c)]{GoodZh} that $ \Ext_H^1({_Hk},{_Hk})\neq 0.$

The algebras in this classification which satisfy a polynomial identity (PI) are all easily seen to be commutative-by-finite, as we now indicate. Omitting the 2 group algebras and the rank 2 polynomial algebra, since these have already been discussed, the remaining PI families are as follows, using the numbering scheme of \cite{GoodZh}:
\begin{enumerate}
%\item[(I)] The group algebras over $k$ of $\mathbb{Z}\times \mathbb{Z}$ and $\mathbb{Z} \rtimes \mathbb{Z}$.
%\item[(II)] The enveloping algebras of the two two-dimensional $k$-Lie algebras.
\item[(III)] Hopf algebras $A(\ell, n,q)$, where $\ell$ and $n$ are integers with $\ell >1$ and $n > 0$, and $q$ is a primitive $\ell^{\mathrm{th}}$ root of 1 in $k$. As an algebra  
$$ A(\ell, n , q) \; := \; k\langle x^{\pm 1},y: xy = qyx \rangle, $$
the localised quantum plane,  with $x$ group-like and $\Delta (y) = y \otimes 1 + x^n \otimes y$. These algebras are free of finite rank over the normal commutative Hopf subalgebra $A:=k [x^{\pm \ell}, y^{\ell'} ]$, where $\ell' := \ell/ \gcd (n, \ell)$.

\item[(IV)] Hopf algebras $B(n, p_0, \dots , p_s,q)$, where $s \geq 2$, $n,p_0,\ldots , p_s$ are positive integers with $p_0 \mid n$ and $\{p_i : i \geq 1 \}$ strictly increasing and pairwise relatively prime, and $q$ is a primitive $\ell^{\mathrm{th}}$ root of 1 where $\ell := (n/p_0)p_1 \ldots p_s$. Then $B(n, p_0, \dots , p_s,q)$ is the subalgebra of the localised quantum plane from (III) generated by $\{x^{\pm 1}, y^{m_i} : 1 \leq i \leq s \}$, where $m_i := \Pi_{j \neq i} p_j.$ This is a Hopf subalgebra of $A(\ell,n,q)$, with $x$ group-like and $\Delta(y^{m_i}) = y\otimes 1 + x^{m_i n} \otimes y$. One can easily check that $A := k[ y^{p_1 \cdots p_s}, x^{\pm \ell} ]$ is a normal commutative Hopf subalgebra over which $B(n,p_0,\ldots , p_s,q)$ is a finite module.
\end{enumerate}

We proceed to describe the Hopf duals of these algebras. Once again we leave some calculations to the reader in order to save space; details of the omitted arguments can be found in \cite{Jahn} or \cite{Cou}.

\begin{proposition}\label{GK2duals} Let $k$ be an algebraically closed field of characteristic 0, and let $H$ be a noetherian Hopf $k$-algebra domain of GK-dimension 2 with $\mathrm{Ext}_H(k,k) \neq 0$, as listed above. Then $H$ satisfies $\mathbf{(CoSplit)}$ and $\mathbf{(OrbSemi)}$, and so the Hopf dual $H^{\circ}$ of $H$ has the following form.
\begin{enumerate}
\item[(III)] When $H = A(\ell,n,q)$, let $d:=\mathrm{gcd}(n,\ell), \, \ell':=\ell/d, \, n':=n/d.$ Then
$$ H^{\circ} \; \cong \; (kC_d \otimes T_f(\ell',n',q^{-d})) \# (k[f,f']\otimes k((k,+)\times k^\times)),$$
with Hopf subalgebras $\overline{H} \cong kC_d \otimes T_f(\ell',n',q^{-d})$,
$$ W(H^{\circ}) \; \cong \; (kC_d \otimes T_f(\ell',n',q^{-d})) \# k[f,f']$$
and
$$ \widehat{kG_H} \; \cong \; (kC_d \otimes T_f(\ell',n',q^{-d})) \# k((k,+)\times k^{\times}).$$
\item[(IV)] When $H = B(n, p_0, \dots , p_s,q)$, define $\xi_i := q^{-(n/p_0)m_i}$, for $i = 1, \ldots , s$, so $\xi$ is a primitive $p_i^{\mathrm{th}}$ root of unity. Then 
$$ H^{\circ} \; \cong \; (kC_{n/p_0}\otimes T_f(p_1,p_0p_2\ldots p_s, \xi_1) \otimes \ldots \otimes T_f(p_s,p_0,\xi_s)) \#(k[f,f']\otimes k((k,+)\times k^\times)),$$
with Hopf subalgebras 
$$\overline{H}^{\ast} \; \cong \; kC_{n/p_0}\otimes T_f(p_1,p_0p_2\ldots p_s, \xi_1) \otimes \ldots \otimes T_f(p_s,p_0,\xi_s),$$
$$ W(H^{\circ}) \; \cong \;   (kC_{n/p_0}\otimes T_f(p_1,p_0p_2\ldots p_s, \xi_1) \otimes \ldots \otimes T_f(p_s,p_0,\xi_s)) \# k[f,f']$$
and 
$$ \widehat{kG_H} \; \cong \; (kC_{n/p_0}\otimes T_f(p_1,p_0p_2\ldots p_s, \xi_1) \otimes \ldots \otimes T_f(p_s,p_0,\xi_s)) \# k((k,+)\times k^\times). $$
\end{enumerate}
\end{proposition}
\begin{proof}(III) As noted above, $A:=k [x^{\pm \ell}, y^{\ell'} ]$ is a normal commutative Hopf subalgebra of $H$. It is clear that $H$ is a free right $A$-module on the basis $\{y^i x^j : 0 \leq i \leq \ell', \, 0 \leq j < \ell \}$. By a straightforward analysis detailed in \cite[Example 1.1.21 and $\S$2.2.6]{Cou}, as algebras
$$ \overline{H} := H/A^+ H \; \cong \; T_f(\ell', n', q^{-d}) \#_{\sigma} kC_d ,$$
for a cocycle $\sigma$ which is in general non-trivial; and, as coalgebras,
\begin{equation}\label{cotense}  \overline{H}  \; \cong \; T_f(\ell', n', q^{-d}) \otimes  kC_d .
\end{equation}
Since $\mathrm{gcd}(n',\ell') = 1$, both tensorands in (\ref{cotense}) are self-dual (using \cite[Exercise 7.3.4]{Radbook} for the left-hand one), and hence
\begin{equation}\label{uno} \overline{H}^{\ast} \; \cong \;  kC_d \otimes T_f(\ell', n', q^{-d})  
\end{equation}
as algebras. Moreover $A \cong \mathcal{O}((k,+)\times k^{\times})$, so that, as Hopf algebras,
\begin{equation}\label{duo} A^{\circ} \; \cong \; k[f,f'] \otimes k((k,+) \times k^{\times}), 
\end{equation}
as recalled in $\S$\ref{commcase}.

Let $X$ be the right $A$-submodule of $H$ generated by $ \lbrace y^ix^j : 1\leq i<\ell', 0\leq j<\ell \rbrace \cup \lbrace x^j-1: 1\leq j<\ell \rbrace$. Clearly $H = A \oplus X$, and it is routine to check that $X$ is a coideal of $H$. Therefore $A \subseteq H$ satisfies $\mathbf{(CoSplit)}$. 

Thus Theorem \ref{decompdual}(1) applies, and with (\ref{uno})and (\ref{duo}) it yields the isomorphism of algebras 
\begin{eqnarray*} H^{\circ} \; &\cong& \; \overline{H}^{\ast}\# A^{\circ}\\
&\cong& \;  (kC_d \otimes T_f(\ell', n', q^{-d})) \# (k[f,f'] \otimes k((k,+)\times k^{\times})),
\end{eqnarray*}
with Hopf subalgebra $ (kC_d \otimes T_f(\ell', n', q^{-d}))$.

It is routine to check that the adjoint action of $\overline{H}$ on $A$ factors through the group algebra $kC_{\ell} = k\langle \overline{x} \rangle$, so that $A \subseteq H$ satisfies 
$\mathbf{(OrbSemi)}$. Therefore Theorem \ref{crux}(2) yields the remaining parts of (III).

(IV) With $A = k[y^{p_1 \cdots p_s}, x^{\pm \ell}]$, a commutative normal Hopf subalgebra of $H$, one can check that $H$ is a free right $A$-module on the basis $\{ y_1^{i_1}\ldots y_s^{i_s}x^j : 0\leq j<\ell, 0\leq i_t<p_t \} $. The structure of the quotient Hopf algebra $\overline{H} := H/A^+ H$ has been analysed in \cite[Lemma 2.2.5]{Cou}. As an algebra it is an iterated crossed product whose detailed description we do not need here; but as a coalgebra it decomposes as
$$ \overline{H} \; \cong \; T_f(p_1,p_0, \xi_1) \otimes T_f(p_2, p_0p_1, \xi_2)\otimes \cdots \otimes T_f(p_s,p_0p_1 \cdots p_{s-1}, \xi_s) \otimes kC_{n/p_0}.$$
Since the above tensorands  are all self-dual because the $p_i$ are mutually coprime \cite[Exercise 7.4.3]{Radbook}, we deduce that, as an algebra,
\begin{equation}\label{hip} \overline{H}^{\ast} \; \cong \; T_f(p_1,p_0, \xi_1) \otimes T_f(p_2, p_0p_1, \xi_2)\otimes \cdots \otimes T_f(p_s,p_0p_1 \cdots p_{s-1}, \xi_s) \otimes kC_{n/p_0}.
\end{equation} 

Just as in (III), 
\begin{equation}\label{hap} A^{\circ} \; \cong \; k[f,f'] \otimes k((k,+) \times k^{\times}).
\end{equation}
Define $X$ to be the right $A$-module generated by 
$$ \mathcal{B} \; := \; \lbrace y_1^{i_1}\ldots y_s^{i_s}x^j : 0\leq j<\ell, 0\leq i_t<p_t, \text{ some } i_t\geq 1 \rbrace \cup \lbrace x^j-1: 1\leq j<\ell \rbrace.$$
Then $X$ is a free $A$-module on the basis $\mathcal{B}$ with $H = A \oplus X$. It is routine to confirm that $X$ is a coideal of $H$, so that $A \subseteq H$ satisfies $\mathbf{(CoSplit)}$.  Therefore Theorem \ref{decompdual}(1) applies, implying that 
\begin{equation}\label{her} H^{\circ} \; \cong \; \overline{H}^{\ast} \#A^{\circ}.
\end{equation}
Combining (\ref{hip}), (\ref{hap}) and (\ref{her}) gives the first two isomorphisms in (IV). 

Finally, note that the adjoint action of each of the elements $y_i$ on $A$ is  trivial. Therefore the action of $\overline{H}$ on $A$ factors through the group algebra $kC_{\ell} = k\langle \overline{x}\rangle$. Hence $A \subseteq H$ satisfies $\mathbf{(OrbSemi)}$, and so the final two isomorphisms in (IV) follow from Theorem \ref{crux}(2).
\end{proof}

\begin{Rmk}\label{further} For the algebras in Proposition \ref{GK2duals}, further information on the definition of the generators of $H^{\circ}$ is obtained in \cite[Corollary 4.4.8 and Appendix A.3]{Cou}.
\end{Rmk}


\begin{thebibliography}{9}

\bibitem{AndruGar}
	N. Andruskiewitsch \& G. Andr\'es-Garcia,
	\emph{Quantum subgroups of a simple quantum group at roots of 1},
	Compositio Math. 145 (2009), 476-500.

\bibitem{AndruSchn}
	N. Andruskiewitsch \& H.-J. Schneider,
	\emph{Finite Quantum Groups and Cartan Matrices},
	Adv. Math. 154 (2000), 1–45. 


\bibitem{BrGoodbook}
	K.A. Brown \& K.R. Goodearl,
	\emph{Lectures on Algebraic Quantum Groups},
	Advanced Courses in Mathematics - CRM Barcelona,
	Birkh\"{a}user, 2002.

\bibitem{BrCou}
	K.A. Brown \& M. Couto,
	\emph{Affine commutative-by-finite Hopf algebras},
       J. Algebra 573 (2021), 56–94.


\bibitem{BrZh}
	K.A. Brown \& J.J. Zhang,	
	\emph{Prime Regular Hopf Algebras of GK-dimension one},
	Proc. Lond. Math. Soc. 101  (2010),  260–302.

\bibitem{BrZh2}
	K.A. Brown \& J.J. Zhang,
	\emph{Survey on Hopf algebras of GK-dimesnion 1 and 2},
	arXiv 2003.14251, to appear Proc. Nanjing Conf., Contemp. Math., Amer. Math. Soc., 2021.

\bibitem{ChPr}
	V. Chari \& A. Pressley,
	\emph{A Guide to Quantum Groups},
	Cambridge University Press, 1994.

\bibitem{Chin}
	W. Chin,
	\emph{Spectra of smash products},
	Israel J. Math. \textbf{72} (1990), 84-98.

\bibitem{ConLy}
	C. De Concini and V. Lyubashenko,
	\emph{Quantum function algebras at roots of 1},
	Advances in Math. 108 (1994), 205–262.

\bibitem{ConProc}
	C. De Concini \& C. Procesi,
	\emph{Quantum Groups}
	in \emph{$D$-modules, Representation Theory, and Quantum Groups} (edited by G. Zampieri \& A. D'Agnolo),
	Lecture Notes in Mathematics 1565, (1993), Springer-Verlag, Berlin, 31-140.

\bibitem{Cou}
	M. Couto,
	\emph{Commutative Hopf algebras and their Finite Dual}
	Ph.D. thesis, University of Glasgow (2019),
	http://theses.gla.ac.uk/74418/ .
	

\bibitem{Wu}
	N. Ding, J. Wu, \& G. Liu, 
	\emph{Classification of affine prime regular Hopf algebras of GK-dimension one},
	Advances in Math. 296 (2014), 1-54.

\bibitem{Dix}
	J. Dixmier,
	\emph{Enveloping Algebras},
	Graduate Studies in Mathematics, Volume 11, 1977.

\bibitem{DT}
	Y. Doi \& M. Takeuchi,
	\emph{Cleft comodule algebras for a bialgebra},
	Comm. Algebra 14 (1986), 801-817.

\bibitem{EW}
	P. Etingof \& C. Walton,
	\emph{Semisimple Hopf Actions on Commutative Domains},
	Advances in Math. \textbf{251} (2014), 47-61.

\bibitem{GL}
	F. Ge \& G. Liu,
	\emph{A combinatorial identity and the finite dual of infinite dihedral group algebra}
	 Mathematika 67 (2021), no. 2, 498–513.

\bibitem{GoodZh}
	K.R. Goodearl \& J.J. Zhang,
	\emph{Notherian Hopf algebra domains of Gelfand-Kirillov dimension two},
	J. Algebra 324 (2010), 3131-3168.

\bibitem{Jac}
	N. Jacobson,
	\emph{A note on Lie algebras of characteristic p},
	Amer. J. Math., 74, No. 2, (1952), 357-359.

\bibitem{Jahn}
	A. Jahn,
	\emph{The finite dual of crossed products},
	Ph.D. thesis, University of Glasgow, December 2014,
	http://theses.gla.ac.uk/6158/ .


\bibitem{KL}
	G. Krause \& T.H. Lenagan,
	\emph{Growth of Algebras and Gelfand-Kirillov Dimension},
	Graduate Studies in Mathematics 22,
	Amer. Math. Soc., Providence, RI, 2000.

\bibitem{LL}
	K. Li \& G. Liu,
\emph{The finite duals of affine prime regular Hopf algebras of GK-dimension one},
	arXiv 2103.00495v2.
	
\bibitem{LR1}
	R.G. Larson and D.E. Radford,
	\emph{Finite dimensional cosemisimple Hopf algebras in characteristic 0 are semisimple},
	J. Algebra 117 (1988), 267-289.


\bibitem{Liu}
	G. Liu,
	\emph{A classification result on prime Hopf algebras of GK-dimension one},
	J. Algebra 547 (2020), 579–667.

\bibitem{LWZ}
	D.-M. Lu, Q.-S. Wu \& J.J. Zhang,
	\emph{Homological integral of Hopf algebras},
	Trans. Amer. Math. Soc. 359 (2007), 4945-4975.



\bibitem{McRob}
	J.C. McConnell and J.C. Robson,
	\emph{Noncommutative Noetherian Rings},
	John Wiley \& Sons, 1987.

\bibitem{Mont}
	S. Montgomery,
	\emph{Hopf Algebras and Their Actions on Rings},
	Amer. Math. Soc., 1992.



\bibitem{Rad80}
	D. E. Radford,
	\emph{On an analog of Lagrange's Theorem for commutative Hopf algebras},
	Proc. Amer. Math. Soc., 79 (1980), 164-166.

\bibitem{Radbook}
	D. E. Radford,
	\emph{Hopf Algebras},
	Series on Knots and Everything, Vol. 49,
	World Scientific Publishing, 2012.


\bibitem{Schau}
	P. Schauenburg,
	\emph{A generalization of Hopf crossed products},
	Comm. in Algebra, 27 (1999), 4779-4801.

\bibitem{Schn92}
	H.-J. Schneider,
	\emph{Normal basis and transitivity of crossed products for Hopf
algebras},
	J. Algebra, 152 (1992), 289-312.

\bibitem{Schn}
	H.-J. Schneider,
	\emph{Some remarks on exact sequences of quantum groups},
	Comm. in Algebra, 21 (1993), 3337-3357.



\bibitem{Sk06}
	S. Skryabin,
	\emph{New Results on the bijectivity of antipode of a Hopf algebra},
	J. Algebra 306 (2006), 622-633.

\bibitem{Sk10}
	S. Skryabin,
	\emph{Hopf Algebra Orbits on the Prime Spectrum of a Module Algebra},
	Algebra Rep. Theory, \textbf{13} (2010), 1-31.


\bibitem{SSW}
	L.W. Small, J.T. Stafford \& R.B. Warfield Jr.,
	\emph{Affine algebras of Gelfand-Kirillov dimension one are PI},
	Math. Proc. Cambridge Philos. Soc. 97 (1985), 407-414.

\bibitem{Tak}
	M. Takeuchi,
	\emph{Free Hopf algebras generated by coalgebras},
	J. Math. Soc. Japan, 23 (1971), 561-582.

\bibitem{WZZ}
	D.-G. Wang, J.J. Zhang \& G. Zhuang,
	\emph{Hopf algebras of GK-dimension two with vanishing Ext-group},
	J. Algebra, 388 (2013), 4963-4986.

\bibitem{Water}
	W. C. Waterhouse,
	\emph{Introduction to Affine Group Schemes},
	Springer-Verlag New York Inc, 1979.




\end{thebibliography}
\end{document}